\pdfoutput=1
\RequirePackage[l2tabu, orthodox]{nag}

\documentclass[reqno]{amsart}

\usepackage{lmodern}
\usepackage[normalem]{ulem}
\usepackage[T1]{fontenc}
\usepackage[utf8]{inputenc}
\usepackage[english]{babel}
\usepackage{microtype} %Better font spacing

\usepackage{amsmath,amssymb,amsthm,mathrsfs,listings,
latexsym,mathtools,mathdots,booktabs,enumitem,tikz,bm,url}
\usepackage[centertableaux]{ytableau}
\usepackage{multirow}

\usetikzlibrary{arrows,positioning,shapes.geometric}
\usepackage[pdftex]{hyperref}
\usepackage[capitalize,noabbrev]{cleveref}

\usepackage{todonotes}
\presetkeys%
    {todonotes}%
    {inline,backgroundcolor=yellow}{}

\newtheorem{theorem}{Theorem}
\newtheorem{proposition}[theorem]{Proposition}
\newtheorem{lemma}[theorem]{Lemma}

\newtheorem{conjecture}[theorem]{Conjecture}

\theoremstyle{definition}
\newtheorem{definition}[theorem]{Definition}
\newtheorem{example}[theorem]{Example}
\newtheorem{remark}[theorem]{Remark}
\newtheorem{question}[theorem]{Question}
\newtheorem{problem}[theorem]{Problem}

\definecolor{lightblue}{rgb}{0.8,0.8,1.0}
\definecolor{lightgreen}{rgb}{0.8,1.0,0.8}
\definecolor{pBlue}{RGB}{86,139,190}
\definecolor{pCyan}{RGB}{149,186,201}
\definecolor{pSand}{RGB}{184,166,121}
\definecolor{pAlgae}{RGB}{87,115,135}
\definecolor{pSkin}{RGB}{236,216,167}
\definecolor{pGray}{RGB}{156,175,156}
\definecolor{pPink}{RGB}{215,114,127}
\definecolor{pOrange}{RGB}{211,153,80}

\newcommand{\defin}[1]{%
\relax\ifmmode%
\textcolor{blue}{#1}%
\else\textcolor{blue}{\emph{#1}}%
\fi%
}
\newcommand{\oeis}[1]{\href{http://oeis.org/#1}{#1}}
\newcommand{\thsup}{\textnormal{th}}

\newcommand{\setN}{\mathbb{N}}

\newcommand{\setR}{\mathbb{R}}
\newcommand{\setC}{\mathbb{C}}
\newcommand{\setH}{\mathcal{H}}

\newcommand{\uvec}{\mathbf{u}}
\newcommand{\vvec}{\mathbf{v}}
\newcommand{\wvec}{\mathbf{w}}
\newcommand{\xvec}{\mathbf{x}}
\newcommand{\yvec}{\mathbf{y}}
\newcommand{\zvec}{\mathbf{z}}

\newcommand{\ee}{\mathbb{E}}

\newcommand{\Bh}{\hat{B}}
\newcommand{\bh}{\hat{b}}

\newcommand{\symS}{S}

\DeclareMathOperator{\des}{des}

\DeclareMathOperator{\length}{\ell}

\DeclareMathOperator{\inv}{inv}
\DeclareMathOperator{\maj}{maj}

\DeclareMathOperator{\peaks}{peaks}
\DeclareMathOperator{\runsort}{runsort}
\DeclareMathOperator{\revflip}{revflip}

\DeclareMathOperator{\insStay}{Stay}
\DeclareMathOperator{\insSwap}{Swap}

\DeclareMathOperator{\insFlip}{Flip}

\newcommand{\slopeSet}{\mathcal{S}}
\DeclareMathOperator{\peakBij}{\mathcal{B}}
\DeclareMathOperator{\peakBijLex}{\mathcal{C}}

\DeclareMathOperator{\spToRSP}{\mathtt{sptorsp}}

\newcommand{\RSW}{\mathrm{RSW}}
\newcommand{\PEAKVAL}{\mathrm{PKV}}
\newcommand{\PKacb}{\mathrm{PKV}^{\ast}}
\newcommand{\SPV}{\mathrm{SPV}}
\newcommand{\DB}{\mathrm{DB}} %descent-bottoms
\newcommand{\LTRMIN}{\mathrm{LRMin}} %Left-to-right-min

\newcommand{\RSP}{\mathcal{RSP}}

\newcommand{\DES}{\mathrm{DES}}
\newcommand{\MIP}{\mathrm{MIP}}
\newcommand{\SYT}{\mathcal{SYT}}
\newcommand{\BW}{\mathcal{BW}} %Biwords
\newcommand{\SSYT}{\mathcal{SSYT}}
\newcommand{\partitionP}{\mathcal{P}}
\newcommand{\setPart}{\mathcal{SP}} % Set-partitions

\newcommand{\eeSet}{\mathcal{E}}

\newcommand{\polySeqV}{\mathcal{V}}

\newcommand{\intLacSec}{\mathcal{F}}

\newcommand{\interl}{\ll}

\tikzset{every picture/.append style={
	scale=1,
	baseline=(current bounding box.center),
	x=1em,
	y=1em,
	thinLine/.style={line width=0.8pt},
	thickLine/.style={line width=1.5pt,line join=round},
	entries/.style={xshift=-0.5em,yshift=-0.5em,font=\small,scale=1.0}
	}
}

\newcommand{\runBlock}[3]{%
\begin{tikzpicture}
\begin{scope}
	\draw[black,x=1em,y=1em,fill=pSkin,line width=1.0pt] (0,0)--(0,1.2)--(5,1.2)--(5,0)--(0,0);
	\node[entries] (A) at (1.5, 1.1) {$#1$};
	\node[entries] (B) at (3,   1.1) {$#2$};
	\node[entries] (C) at (4.5, 1.1) {$#3$};
\end{scope}
\end{tikzpicture}
}

\newcommand{\shortBlock}[2]{%
\begin{tikzpicture}
\begin{scope}
	\draw[black,x=1em,y=1em,fill=pSkin,line width=1.0pt] (0,0)--(0,1.2)--(3,1.2)--(3,0)--(0,0);
	\node[entries] (A) at (1.5,1.1) {$#1$};
	\node[entries] (B) at (2.5,1.1) {$#2$};
\end{scope}
\end{tikzpicture}
}

\newcommand{\singleBlock}[1]{%
\begin{tikzpicture}
\begin{scope}
	\draw[black,x=1em,y=1em,fill=pSkin,line width=1.0pt] (0,0)--(0,1.2)--(1,1.2)--(1,0)--(0,0);
	\node[entries] (A) at (1,1.1) {$#1$};
\end{scope}
\end{tikzpicture}
}

\title{Peaks are preserved under run-sorting}

\author[Per Alexandersson]{Per Alexandersson}
\address{Department of Mathematics, Stockholm University, SE-106 91 Stockholm,
      Sweden}
\email{per.w.alexandersson@gmail.com}

\author{Olivia Nabawanda}
\address{Department of Mathematics, Makerere University, Kampala, Uganda}
\email{onabawanda@must.ac.ug}

\begin{document}
\begin{abstract}
We study a sorting procedure (run-sorting) on permutations, where runs are rearranged in lexicographic order.
We describe a rather surprising bijection on permutations on length $n$,
with the property that it sends the set of peak-values to the set of peak-values after run-sorting.
We also prove that the expected number of descents in
a permutation $\sigma \in S_{n}$ after run-sorting is equal to $(n-2)/3$.
Moreover, we provide a closed form of the exponential generating function introduced by
Nabawanda, Rakotondrajao and Bamunoba in 2020, for the number of run-sorted permutations
of $[n]$, ($\RSP(n)$) having $k$ runs, which gives a new interpretation to the
sequence \oeis{A124324} in the Online Encyclopedia of Integer Sequences. 
We show that the descent generating polynomials, $A_{n}(t)$ for $\RSP(n)$ are real rooted,
and satisfy an interlacing property similar to that satisfied by the Eulerian polynomials.

Finally, we study run-sorted binary words and compute the expected number
of descents after run-sorting a binary word of length $n$.
\end{abstract}
\maketitle 
\section{Introduction and preliminaries}

Studying the distribution of different permutation statistics such as runs, descents,
ascents, excedances and peaks among others has a long history and has been an
area of active research in enumerative combinatorics and mathematics at large,
see e.g.,~\cite{Liagre1855, WarrenSeneta1996, EhrenborgSteingrismsson2000, MacMahon1960, MantaciRakotondrajao2003, Kitaev2007}.
Among the class of permutations where these statistics have been investigated are the so called \emph{flattened partitions}
introduced by D.~Callan, see \cite{Callan2009}. Callan borrowed the word flatten from a function in Mathematica,
where \texttt{Flatten} concatenates a list of lists into one single list. 

Flattened partitions have received a lot of attention since then,
see e.g.,~\cite{MansourShattuckWagner2015,LiuZhang2015, MansourShattuck2011,
NabawandaRakotondrajaoBamunoba2020, NabawandaRakotondrajao2020x}. 
Our main source of inspiration for this paper
is the recent study on flattened partitions by
Nabawanda, Rakotondrajao and Bamunoba~\cite{NabawandaRakotondrajaoBamunoba2020}.

Given a non-empty finite subset $T$ of positive integers, a \defin{set partition} $P$ of $T$ is
a collection of disjoint non-empty subsets $B_1, B_2, \dotsc, B_k$ of $T$ (which we simply call \defin{blocks})
such that $B_1 \cup B_2 \cup \dotsb \cup B_k = T$.
To generate a flattened partition from $P$, the elements in each of the blocks of $P$
are first sorted increasingly, and then the blocks are arranged in lexicographical order.
For instance, $\sigma = 15672835$ is flattened as it arises from the set partition $1567|28|35$,
but $\sigma' = 4516732$ is not a flattened partition.
Flattened partitions constitute a subset of $\defin{\symS_{n}}$, the set of all permutations
of $\{1, 2, \dotsc, n\}$. 
To be consistent, we shall instead refer to \defin{flattened partitions} as \defin{run-sorted permutation},
since we also consider run-sorted binary words in \cref{D}.
We let \defin{$\RSP(n)$} denote the set of run-sorted permutations of length $n$.

For a fixed positive integer $n$, we set $\defin{[n]} \coloneqq \{1, 2, \dotsc, n\}$.
A permutation $\sigma \in \symS_{n}$ will be represented in the one-line notation, $\sigma(1)\sigma(2)\dotsm\sigma(n)$.
In particular, every permutation can be considered as a word of length $n$, with letters in $\setN$.

We say that the word $w$ (of length $n$) has $k$ as a \defin{descent} if $w(k) > w(k+1)$,
where $k \in [n-1]$, where $n$ is the length of $w$.
The \defin{descent set} of $w$ is denoted by $\defin{\DES(w)}$ and $\defin{\des(w)}$
is the cardinality of $\DES(w)$.
A closely related statistic is the set of \defin{descent bottoms};
\begin{equation}
 \defin{\DB(\pi)} \coloneqq \left\{ w_{i+1} : i \in [n-1] \text{ and } w_{i}>w_{i+1} \right\}.
\end{equation}
The set of \defin{left-to-right minima} is defined as
\begin{equation}\label{eq:ltrminDef}
  \defin{\LTRMIN(w)} \coloneqq \left\{ w_i : w_i = \min \{w_1, w_2, \dotsc, w_i \} \right\}.
\end{equation}

\medskip 

A word $w$ with exactly $r-1$ descents can be decomposed into $r$ (weakly) increasing subwords.
Each of these subwords is a \defin{run} of $w$.
For $w = 1526734$ (which is a permutation),
we have the descents $2$ and $5$, and the runs of $w$ are $15$, $267$ and $34$.
Given a word $w$, we let $\defin{\runsort(w)}$ be the word obtained
by rearranging the runs of $w$ in lexicographic order. 
For example,
\begin{align*}
 \runsort(29\; 7\; 368\; 5\; 14) &= 14\; 29\; 368\; 5\; 7 \\
 \runsort(1\;011 \; 0111 \; 00011) &=  00011 \; 011 \; 0111 \; 1.
\end{align*}
In particular, if $\sigma \in \symS_n$, then $\runsort(\sigma) \in \RSP(n)$.
\medskip 

In \cref{A}, \cref{thm:spToRSP}, we refine a result by
Nabawanda, Rakotondrajao and Bamunoba~\cite[Sec. 4]{NabawandaRakotondrajaoBamunoba2020},
which describes a bijection between set partitions of $[n]$ and run-sorted permutation of size $n+1$.
Our refinement allows us to compute the descent set of the resulting permutation from
the set partition.

In \cref{B}, we describe the multivariate generating function for the descent set distribution
of run-sorted permutation $\sigma$.
This solves the problem in \cite{NabawandaRakotondrajaoBamunoba2020}, where an
exponential generating function for the number of flattened partitions
was described implicitly as a solution to a partial differential equation.
We discover that our explicit exponential generating function 
produces a Sheffer sequence.
We also show in \cref{thm:interlacing} that the polynomials 
\[
A_{n}(t) \coloneqq \sum_{\sigma \in \RSP(n)} t^{\des(\sigma)}
\]
studied earlier in \cite{NabawandaRakotondrajaoBamunoba2020} are real-rooted.
In fact, we show the stronger property that the roots of $A_{n-1}(t)$ \emph{interlace} the roots of $A_{n}(t)$. 
\begin{definition}
Let $f$ and $g$ be polynomials with real, non-positive
roots $f_{i}$ and $g_{i}$, respectively.
We say that $f$ \defin{interlaces} $g$, and we write $\defin{f \preceq g}$ if
\[
\dotsm \leq f_{3}\leq g_{3}\leq f_{2}\leq g_{2}\leq f_{1}\leq g_{1} \leq 0.
\]
Note that $\deg(f)=\deg(g)$ or $\deg(f)+1=\deg(g)$.
\end{definition}
\medskip

A \defin{peak} of a permutation $\sigma \in \symS_n$, is an integer $i$, $1 < i < n$ such
that $\sigma(i-1)<\sigma(i)>\sigma(i+1)$. If $i$ is a peak of $\sigma$,
we say that $\sigma(i)$ is a \defin{peak-value} of $\sigma$.
For example, the permutation $\sigma = 1374625 \in \symS_{7}$ has peak-values $7$ and $6$ with
peaks $3$ and $5$ respectively.
Let $\defin{\PEAKVAL(\sigma)}$ denote the set of peak-values of the permutation $\sigma$,
and set $\defin{\SPV(\sigma)} \coloneqq \PEAKVAL(\runsort(\sigma))$.

The set of \defin{$132$-peak-values} is defined as
\[
 \defin{\PKacb(\pi)} \coloneqq \left\{ \pi_i : 1<i<n \text{ and } \pi_{i-1}<\pi_{i+1}<\pi_{i} \right\}.
\]
Note that $\PKacb(\pi) \subseteq \PEAKVAL(\pi)$.

We shall also use the notation $\defin{\xvec_T}\coloneqq x_{t_1}x_{t_2}\dotsm x_{t_m}$
whenever $T = \{t_1,t_2,\dotsc,t_m\}$ is a set of positive integers.

In \cref{C}, we prove the theorem below, which is the main result of this article.
\begin{theorem}
For any positive integer $n \geq 1$, then 
\[
\sum_{\sigma \in \symS_{n}} \xvec_{\PEAKVAL(\sigma)} = 
\sum_{\tau \in \symS_{n}}  \xvec_{\SPV(\tau)}.
\]
\end{theorem}
In fact, we prove a stronger version of this identity
where an additional set-valued statistic is 
involved. This is achieved through a bijective proof, which exploits the 
recursive process of constructing a permutation in $\symS_n$
from a permutation in $\symS_{n-1}$ by inserting $n$ somewhere.
We also prove that the expected number of
descents\footnote{We show that descents are always peaks in a run-sorted permutation.}
in a run-sorted permutation $\sigma$ is given by
\[
\ee[\des(\runsort(\sigma))] = \ee[\peaks(\sigma)] = \frac{n-2}{3}.
\]

In \cref{D}, we prove that the number of
permutations $\sigma \in \symS_{a+b}$ with major index $a$ and $b$ inversions,
can be set in bijection with run-sorted binary words having $a$ $0$'s and $b$ $1$'s.
We are also able to compute the expected number of descents in $\runsort(w)$,
for $w$ being a uniformly chosen binary word of length $n$, see \cref{ssec:BWEEDes}.

\section{Set partitions of \texorpdfstring{$[n]$}{[n]} 
and the descent set in a run-sorted permutation of \texorpdfstring{$[n+1]$}{[n+1]}}\label{A}

Let \defin{$\setPart(n)$} be the set of set-partitions of $[n]$.
Each element in $\setPart(n)$ is expressed as $B_{1}|B_{2}|\dotsb |B_{m}$,
where the elements in each block are written in increasing order
and the blocks arranged in lexicographic order.

\begin{definition}
We shall now recall a bijection $\defin{\spToRSP}: \setPart(n) \rightarrow \RSP(n+1)$ 
from \cite[Sec, 4]{NabawandaRakotondrajaoBamunoba2020}.
Given $P \in \setPart(n)$, we first move the smallest element in
each block to the end of the block. Then, all entries are increased by one,
and we remove the vertical bars separating the blocks. This 
creates a word of length $n$. Finally, we append a $1$ to the beginning of the word.
The result is now an element in $\RSP(n+1)$.
\end{definition}

\begin{example}
Consider the set partition $P = 1258|3|47|6$.
Moving the smallest element in each block to the end of the block,
then increasing the entries by $1$, and finally removing the vertical
bars produces the word $w = 36924857$, of length $8$.
Inserting $1$ at the beginning of $w$ gives $136924857 \in \RSP(9)$.
\end{example}

In \cite{NabawandaRakotondrajaoBamunoba2020}, the following theorem was proved.
\begin{theorem}[{See \cite[Thm. 19]{NabawandaRakotondrajaoBamunoba2020}}]\label{thm:spToRSP}
For any integer $n \geq 0$, if $P$ is a set-partition of $[n]$
and $\sigma = \spToRSP(P)$, then the following assertions are equivalent:
\begin{enumerate}
		\item [(i)] the number of blocks of size greater than $1$ of the partition $P$ is $k-1$,
		\item [(ii)] the number of runs in the flattened partition $\sigma$ is equal to $k$.
\end{enumerate}
\end{theorem}

We shall refine \cref{thm:spToRSP} by describing
how to obtain the set $\DES(\sigma)$ from the set partition corresponding to $\sigma$.
\begin{proposition}\label{prop:DESfromSP}
Suppose $\sigma = \spToRSP(P)$ for some $P = B_1 | B_2 |\dotsb | B_m \in \setPart(n)$.
For all $1 \leq i \leq m$, let $p_{i}$ denote the position
of the last element in $B_{i}$.
Then
\[
  \DES(\sigma) = \{ p_{i}:  i \in [m] \text{ where } |B_{i}|\geq 2 \}.
\]
\end{proposition}
\begin{proof}
By \cref{thm:spToRSP}, if there are $k$ blocks of size at least $2$, we have that $|\DES(\sigma)| = k$,
and $\sigma = \spToRSP(P)$ has a total of $k+1$ runs.
The lengths $l_{i}$, of the first $k$ runs of $\sigma$ are given by
\[
l_{i} = p_{i} - p_{i-1},
\]
with the convention $p_{0} \coloneqq 0$.
Moreover, $\sigma(p_{i})$ is given by the relation
\[
\sigma(p_{i}) = p_{1}+ n-i,
\]
and we also have that $\sigma(1) = 1$.
The remaining elements of $\sigma$ are distributed from the
set $\{2, 3, \dotsc, \sigma(p_{k})-1\}$ in the same order,
such that the lengths of the first $k$ runs are given by the values of the $l_{i}$.
From these observations, we can draw the desired conclusion.
\end{proof}

\begin{remark}\label{rem:desArePeaks}
We should also note that the set of of peaks is the same as set of descents,
for any element in $\RSP(n)$.
Suppose $\sigma \in\RSP(n)$ and $k$ is a descent of $\sigma$,
so that $\sigma(k)>\sigma(k+1)$.
Now, since $\sigma$ is run-sorted, we must have that $\sigma(k-1)<\sigma(k)$,
otherwise, $\sigma(k)$ is a run of length $1$ in $\sigma$,
violating \cref{prop:DESfromSP}.
Hence, $k$ is also a peak of $\sigma$.
\end{remark}

\begin{example}
Let us consider a set partition $P = 18|27|3|46|5$ of $[8]$
having $3$ blocks of size $\geq 2$ and let $\sigma = \spToRSP(P)$.
Then $\DES(\sigma) = \{2, 4, 7\}$. 
The lengths $l_{i}$ for $1 \leq i \leq 3$ of the first $3$ runs of $\sigma$
are  $l_{1} = 2-0 = 2, l_{2} = 4-2 = 2,$ and $l_{3} = 7-4 = 3$.

The last elements in each of the first $3$ runs of $\sigma$,
corresponding to each position in $\DES(\pi)$ are $\sigma(p_{1}) = 2+8-1 = 9,
\sigma(p_{2}) = 2+8-2 = 8,$ and $\sigma(p_{3}) = 2+8-3 = 7$.

Since $\sigma$ is run-sorted, $\sigma(1) = 1$, and its remaining elements
are chosen from the set $\{1, 2, \dotsc, 6\}$ and placed in this order.
Consequently, $\sigma = 19\; 28 \;347\; 56$.
\end{example}

\section{Run-sorted permutations and descent sets}\label{B}

In this section, we describe multivariate generating functions 
for the descent set distribution of run-sorted permutations.
The main result in this section is \cref{thm:genFuncForRSPDes}, where 
we provide the exponential generating function for
flattened partitions of size $n$,
where the number of descents is tracked.
This resolves an open problem in \cite{NabawandaRakotondrajaoBamunoba2020},
and we discover an interesting connection with Sheffer sequences.
\medskip 

We start with some recursions for tracking the descent set.
\begin{theorem}\label{thm:descentSetRecursion}
For all integers $n \geq 1$, let
\[
 A_n(\xvec) \coloneqq \sum_{\pi \in \RSP(n)} \prod_{j \in \DES(\pi)} x_{n-j},
\]
so that we keep track of the descent set, \emph{by indexing from the end}.

Then we have the recursive relations
\begin{equation}\label{eq:rspDesRec1}
A_n(\xvec) = 1 + \sum_{i = 1}^{n-2}\left(\binom{n-1}{i}  - 1\right) x_{i} A_i(\xvec)
\end{equation} 
and 
\begin{equation}\label{eq:rspDesReck}
A_n(\xvec) = A_{n-1}(\xvec) + \sum_{i = 1}^{n-2} \binom{n-2}{i-1} x_{i} A_i(\xvec).
\end{equation}
\end{theorem}

Before proving this, we shall look at a small example, because we use 
a somewhat non-standard convention.
\begin{example}
Let us compute the polynomial $A_5(\xvec)$.
Recall that the presence of $x_i$ indicates 
a descent at the $i^\thsup$ position from the end.
In \cref{tab:tc}, we compute the sum over all permutations 
in $\RSP(5)$, weighted by a monomial.
\begin{table}[!ht]
		\centering
		\begin{tabular}{p{7cm}l}
			$\pi \in \RSP(5)$ & $\prod_{j \in \DES(\pi)} x_{n-j}$ \\
			\toprule
			$12345$  & $1$  \\
			$12354$, $12453$, $13452$ & $x_1$  \\
			$12435$, $12534$, $13425$, $13524$, $14523$ & $x_2$ \\
			$13245$, $14235$, $15234$ & $x_3$ \\
			$13254$, $14253$, $15243$ & $x_1x_3$ \\
			\bottomrule
		\end{tabular}
		\caption{The set $\RSP(5)$ arranged according to the descent set.
		By summing the monomials in the second column, we see that $A_5(\xvec)=1 + 3 x_1 + 5 x_2+3 x_3 + 3 x_1 x_3$.
		}\label{tab:tc}
	\end{table}
\end{example}
\begin{proof}[Proof of \cref{thm:descentSetRecursion}]
We begin by proving the first relation.
Note that there is only one run-sorted permutation with zero descents, i.e.,
the identity permutation.
This corresponds to the $1$ in the right-hand side of \eqref{eq:rspDesRec1}.
Now consider a run-sorted permutation $\sigma\in \RSP(n)$
having at least one descent (and thus at least two runs).
Let $i \in \{ 2, 3, \dotsc, n-1\}$ be the largest such 
descent (indexed from the end) meaning that the first run of $\sigma$
has length $n-i$.
Removing the first run of $\sigma$, prepending a $1$, and then standardizing the 
result, produces a $\tau \in \RSP(i)$,
where $\des(\sigma)=1+\des(\tau)$.
Now, the elements greater than $1$ in the first run of $\sigma$,
is a subset of $\{2,\dotsc,n\}$ of cardinality $n-i-1$.
Moreover, any subset here is possible except
\[
 \{2, 3, \dotsc, n-i-1\},
\]
as in this case, there would not be a descent between the first and second run.
It follows that given $\tau \in \RSP(i)$ and one of the $\binom{n-1}{n-i-1} - 1$
subsets of $\{2,\dotsc,n\}$, we can uniquely recover $\sigma$,
by simply inserting the elements in the chosen subset after the $1$ in increasing order,
and then standardizing the remaining entries so that a permutation is produced.
This operation creates a new descent at position $i$ from the end,
explaining the $x_{i}$ factor and we have proved \eqref{eq:rspDesRec1}.
\medskip 

For the identity in \eqref{eq:rspDesReck}, 
we shall describe a bijection, 
that maps $\sigma \in \RSP(n)$ to either 
some  $\tau \in \RSP(n-1)$ or to some $\tau \in \RSP(i)$ for some $i<n$, together with 
a subset of size $n-i-1$ of $\{3,4,\dotsc,n\}$.

\textbf{First case:} $1$ and $2$ are in the same run of $\sigma$.
We remove the $1$ from $\sigma$ and decrease all entries by one.
This produces $\tau \in \RSP(n-1)$, and $\des(\sigma)=\des(\tau)$.

\textbf{Second case:} $1$ and $2$ in different runs of $\sigma$.
The second run of $\sigma$ must then necessarily start with $2$,
and the first run of $\sigma$ contains $1$, 
and some subset of $\{3,4,\dotsc,n\}$ of size $n-i-1$,
for some $i \in [n-2]$. Note that $i$ is the number of elements in
$\sigma$ not in the first run.
We then let $\tau \in \RSP(i)$ be obtained from $\sigma$
by removing the first run, and then standardize the result. 
This preserves all descents (and their indexing from the end),
except the descent between the first and 
second run of $\sigma$, which is located at position $i$ (from the end).

This procedure is invertible --- given $\tau \in \RSP(i)$ and 
one of the $\binom{n-2}{n-i-1}$ subsets of $\{3,4,\dotsc,n\}$, 
there is a unique $\sigma \in \RSP(n)$
where the first run of $\sigma$ consists of $1$ and the elements in the chosen subset.
Hence, we have proved that 
$A_n(\xvec) = A_{n-1}(\xvec) + \sum_{i = 1}^{n-2} \binom{n-2}{n-i-1} x_i A_{i}(\xvec)$,
and by noting that $\binom{n-2}{n-i-1} = \binom{n-2}{i-1}$, we are done.
\end{proof}

\begin{example}
Let us illustrate  the proof of \eqref{eq:rspDesRec1} and construct a run-sorted 
permutations of $[6]$ having two
descents from a run-sorted permutation $\tau$ of $[3]$ having 
one descent at position $1$ from the end.
We have $\tau = 132$. The allowed subsets of size two from the
set $\{2, 3, 4, 5, 6\}$ are 
\[
\{2, 4\}, \{2, 5\}, \{2, 6\}, \{3, 4\}, \{3, 5\}, \{3, 6\}, \{4, 5\}, \{4, 6\}, \{5, 6\}.
\]
From the pair $(\{3, 4\}, \tau)$, we then get $\pi = 134265$, which has a
descent at position $1$  (from the end) 
and the second descent at position $3$ (from the end).
\end{example}

\medskip 
We now turn our attention to the univariate sequence $\{A_n(t)\}_{n\geq 0}$
of polynomials where $A_n(t)$ is obtained from $A_n(\xvec)$ by letting $x_i \to t$.
Equivalently,
\begin{equation}
 \defin{A_n(t)} \coloneqq \sum_{\pi \in \RSP(n)} t^{\des(\pi)}.
\end{equation}
The first few such polynomials are given in \cref{tab:t1}.
From the definition, it is straightforward to see that $A_n(t)$ has degree $\lfloor n/2 \rfloor$.

\begin{table}[!ht]
\begin{tabular}{cr}
$n$   & $A_n(t)$\\
\toprule
$ 1$ & $1$\\
$ 2$ & $1$ \\
$3$ & $t+1$ \\
$ 4$ & $4t+1$ \\
$5$ & $3t^2+11t+1$ \\
$ 6$ & $25t^2+26t+1$\\
$ 7$ & $	15t^3+130t^2+57t + 1$ \\
$8$ & $210 t^3+546 t^2+120 t+1$\\
$ 9$ & $105 t^4+1750 t^3+2037 t^2+247 t+1$\\
\bottomrule
\end{tabular}
\caption{The polynomials $A_n(t)$.}\label{tab:t1}
\end{table}

Now let $\defin{f_{n, k}}$ be the number of run-sorted
permutations of $[n]$ having $k$ runs, see \oeis{A124324} in the OEIS.
Note that 
\begin{equation}\label{eq:fnkDef}
 t A_n(t) = \sum_{\pi \in \RSP(n)} t^{\des(\pi)+1} = 
 \sum_{k\geq 1} t^{k} f_{n,k}.
\end{equation}
In \cite{NabawandaRakotondrajaoBamunoba2020}
the authors proved that for all integers $n$ and $k$ 
such that $1 \leq k \leq n$, we have
\begin{equation}
f_{n+2, k} = f_{n+1, k} + \sum_{i = 1}^{n}\binom{n}{i}f_{n+1-i, k-1}.
\end{equation}
Using this recurrence relation, the exponential generating
function, $\defin{F(t, u)} \coloneqq \sum_{n,k\geq 0} \frac{f_{n,k}}{n!} u^n t^k$
for these numbers was found to satisfy
\begin{equation}
\frac{\partial F(t, u) }{\partial u} = t\exp(t(\exp(u) - 1)+ u(1-t)), \text{ where }
\dfrac{\partial F(t, 0)}{\partial u} = t.
\end{equation}
We observe that $F(t, u)$ is not in its closed form and the authors
of this article left this as an open problem to find a closed-form solution.

We now use the recurrence relation in \eqref{eq:rspDesRec1},
to instead find the generating function for the polynomials $A_{n+1}(t)$,
where we have shifted the index by one. This modification 
allows us to derive a closed form for the corresponding exponential generating function.

\begin{theorem}\label{thm:genFuncForRSPDes}
Let $G(t,u)$ be the exponential generating function 
\[
G(t,u) = \sum_{n \geq 0} A_{n+1}(t) \frac{u^n}{n!}.
\]
Then
 \begin{equation}\label{eq:genFuncForRSPDes}
 G(t,u) = \exp\left[u + t(e^u-u-1) \right].
 \end{equation}
\end{theorem}
\begin{proof}
The recursion in \eqref{eq:rspDesReck} gives that
\[
  \sum_{n \geq 0} A_{n+1}(t) \frac{u^n}{n!} = 
  \sum_{n \geq 0} A_{n}(t) \frac{u^n}{n!} + 
  t \sum_{n \geq 0} \sum_{j=1}^{n-1} A_{j}(t) \frac{(n-1)!}{(j-1)! (n-j)!} \frac{u^n}{n!}.
\]
Rewriting now gives
\[
  G(t,u) = 
  \sum_{n \geq 0} A_{n}(t) \frac{u^n}{n!} + 
  t \sum_{j \geq 1} \sum_{n \geq j+1} A_{j}(t) \frac{(n-1)!}{(j-1)! (n-j)!} \frac{u^n}{n!}.
\]
We differentiate both sides with respect to $u$, and get 
\[
\frac{\partial}{\partial u}
  G(t,u) = 
  \sum_{n \geq 1} A_{n}(t) \frac{u^{n-1}}{(n-1)!} 
  + 
  t \sum_{j \geq 1} \sum_{n \geq j+1} A_{j}(t) \frac{(n-1)!}{(j-1)! (n-j)!} \frac{u^{n-1}}{(n-1)!}.
\]
This leads to 
\[
\frac{\partial}{\partial u} G(t,u) - G(t,u) = 
  t \sum_{j \geq 1} \frac{A_{j}(t) u^{j-1}}{(j-1)!} \sum_{n \geq j+1} \frac{u^{n-j}}{(n-j)!}.
\]
The inner sum is simply $e^u-1$, so
\begin{align*}
\frac{\partial}{\partial u} G(t,u) - G(t,u) &= 
  t(e^u-1) \sum_{j \geq 1} \frac{A_{j}(t) u^{j-1}}{(j-1)!} \\
&= t(e^u-1) G(t,u) .
\end{align*}
Hence, we arrive at the differential equation $\frac{\partial G}{\partial u} = (1 + t(e^u-1))G$.
It is now straightforward to verify that the function in \eqref{eq:genFuncForRSPDes}
is the solution to this differential equation we seek.
\end{proof}

As a final observation, a polynomial
sequence $\{S_{n}(t)\}_{n \geq 0}$ is said to be \defin{Sheffer}
if and only if its generating function has the form
\begin{equation*}
\sum_{n=0}^{\infty}\dfrac{S_{n}(t)}{n!}u^n = P(u) e^{t\cdot Q(u)},
\end{equation*}
where $P(u) = P_{0}+ P_{1}u+ P_{2}u^2+ \dotsb$ and
$Q(u) = Q_{1}(u)+Q_{2}u^2+ \dotsb$ with $P_{0}, Q_{1} \neq 0$, see \cite{Steffensen1941,Shiu1982}.
From the generating function $G(t, u)$ in \eqref{eq:genFuncForRSPDes}
it follows that
the sequence $\{ A_{n+1}(t)/n! \}_{n \geq 0}$ is indeed a Sheffer
sequence with $P(u) = e^u$ and $Q(u) = e^u-u-1$.

\begin{question}
 Can we use properties of Sheffer sequences to deduce additional information
 about the polynomials $\{A_{n}(t)\}_{n \geq 1}$?
 Sheffer sequences are closely related to 
 umbral calculus and shift operators.
 Does this carry over to the combinatorial side here in a meaningful way?
\end{question}

\subsection{Realrootedness and interlacing roots}

It is well-known that the Eulerian polynomials, $E_n(t)$,
defined as
\begin{equation}\label{eq:eulerianDef}
  \defin{E_n(t)} \coloneqq \sum_{\sigma \in \symS_n} t^{\des(\sigma)}
\end{equation}
are all real-rooted. In this subsection, we shall prove that 
our polynomials $A_n(t)$ are also real-rooted.
It will be convenient to work
with $R_n(t) \coloneqq t A_n(t)$ as in \eqref{eq:fnkDef},
and we recall the definition of $f_{n,k}$ from there.
% 
% and define $\defin{f_{n,k}}$ via
% \begin{equation*}\label{eq:fpDesPoly}
% \defin{R_n(t)} \coloneqq \sum_{\pi \in \RSP(n)} t^{\des(\pi)+1} = \sum_{k} t^{k}f_{n, k},
% \end{equation*}
% so that $f_{n,k}$ again denotes the number of run-sorted permutations with $k$ runs.
%
In \cite{NabawandaRakotondrajaoBamunoba2020}, it was proved that
the numbers $f_{n, k}$ satisfy the recurrence relation 
\begin{equation}\label{eq:fnkRec}
f_{n, k} = kf_{n-1, k} + (n-2)f_{n-2, k-1}
\text{ whenever } 1 \leq k < n.
\end{equation}

\begin{lemma}
We have that the $R_n(t)$ satisfy the recurrence
\begin{equation}\label{diff-recurrence}
R_n(t) = tR'_{n-1}(t) + t(n-2)R_{n-2}(t),
\end{equation}
with initial conditions $R_1(t)=R_2(t)=t$.
\end{lemma}
\begin{proof}
By \eqref{eq:fnkRec}, we have that
\begin{align*}
R_n(t) &= \sum_{k} t^{k}\left( kf_{n-1, k} + (n-2)f_{n-2, k-1} \right)\\
&= \sum_{k} t^{k}kf_{n-1, k} + \sum_{k} (n-2)t^{k}f_{n-2, k-1}\\
&= t\sum_{k} k t^{k-1}f_{n-1, k} + t(n-2)\sum_{k-1} t^{k-1}f_{n-2, k-1}.
\end{align*}
This is now recognized as \eqref{diff-recurrence}.
\end{proof}

Our goal is now to use \eqref{diff-recurrence} to show that 
all $R_n(t)$ are real-rooted. We shall in fact prove a much stronger property,
namely that $R_{n-1}(t)$ \emph{interlaces} $R_{n}(t)$ for all $n\geq 1$.
Let $f$ and $g$ be polynomials with real, non-positive
roots $f_{i}$ and $g_{i}$, respectively.
We say that $f$ \defin{interlaces} $g$, and we write $\defin{f \preceq g}$ if
\[
\cdots \leq f_{3}\leq g_{3}\leq f_{2}\leq g_{2}\leq f_{1}\leq g_{1} \leq 0.
\]
Note that $\deg(f)=\deg(g)$ or $\deg(f)+1=\deg(g)$,
and that $0 \preceq t$.

The following lemma is our main tool for proving that our sequence of polynomials are interlacing. 
\begin{lemma}[{D.Wagner, \cite[Sec. 3]{Wagner1992}}]\label{interlacelemma}
Let $f, g, h \in \mathbb{R}[t]$ be real rooted polynomials with
only real, non-positive roots and positive leading coefficients.
Then
\begin{itemize}
\item[(i)] if $f \preceq h$ and $g \preceq h$ then $f+g \preceq h$.
\item[(ii)] if $h \preceq f$ and $h \preceq g$ then $h \preceq f+g$.
\item[(iii)] $g \preceq f$ if and only if $f \preceq tg$.
\end{itemize}
\end{lemma}

We are now ready to prove the main result in this section.
\begin{theorem}\label{thm:interlacing}
The polynomials 
\[
 R_n(t) = \sum_{\pi \in \RSP(n)} t^{\des(\pi)+1}
\]
satisfy $R_{n-1} \preceq R_{n}$ for all $n \geq 1$.
In particular, they are all real-rooted.
\end{theorem}
\begin{proof}
We first note that the statement is true for $n=1$, 
since $R_0(t)=0$ and $R_1(t)=t$.
We shall proceed by induction over $n$,
so fix $n\geq 2$ and assume that the polynomials
$R_{n-1}(t)$ and $R_{n-2}(t)$ are interlacing i.e.,
$R_{n-2} \preceq R_{n-1}$.

It suffices to prove that $R_{n-1} \preceq R_{n}$.
By Rolle's theorem, we know that $R'_{n-1}$
interlaces $R_{n-1}$ i.e., $R'_{n-1} \preceq R_{n}$.
Together with the induction hypothesis 
and (iii) in \cref{interlacelemma},
we have 
\[
tR'_{n-1} \preceq R_{n-1} \text{ and }t(n-2)R_{n-2}\preceq R_{n-1}.
\]
Now, by (ii) in \cref{interlacelemma},
we have that 
\[
R_{n-1} \preceq tR'_{n-1} + t(n-2)R_{n-2},
\]
and by using \eqref{diff-recurrence},
we can now conclude that $R_{n-1} \preceq R_{n}$.
\end{proof} 

The preceding result also implies that $A_{n-1} \preceq A_{n}$
for all $n>1$, as $A_n(t) = R_n(t)/t$.
In particular, all the $A_{n}$ are real-rooted.

\subsection{Conjectures on stability}

Let $\defin{\setH} \subset \setC$ denote the 
upper half-plane $\{z \in \setC: \mathrm{im}(z) >0 \}$.
A multivariate polynomial $P \in \setC[z_1,\dotsc,z_n]$ 
is called \defin{stable} if it does not vanish on $\setH^n$.
That is, $P$ is stable if
\[
 \zvec^* \in \setH^n \implies P(\zvec) \neq 0.
\]
One can easily show that $P$ is stable if and only if 
$P(\alpha+\lambda t)=0$ has only real zeros
for all $\alpha \in \setR^n$, $\lambda \in \setR_+^n$.
Note that a univariate polynomial with real coefficients is stable if and only if 
all roots are real.

The following theorem is attributed to P.~Br{\"a}nd{\'e}n,
see \cite[Thm. 2.5]{HaglundVisontai2012}, and \cite{BrandenHaglundVisontaiWagner2011}.
\begin{theorem}
Let $\tilde{E}_n(\xvec)$ be defined as
\[
 \tilde{E}_n(\xvec) \coloneqq \sum_{\pi \in \symS_n} \prod_{\pi_j > \pi_{j+1}} x_{\pi_j}.
\]
Then $\tilde{E}_n(\xvec)$ is stable.
\end{theorem}
Note that setting $x_i \to t$, we recover the
classical Eulerian polynomials in \eqref{eq:eulerianDef},
so Bränden's result implies that the Eulerian polynomials are real-rooted.

The following notion was introduced in \cite{LeakeRyder2019}
and is a strictly weaker notion than stability.
\begin{definition}[Same-phase stability]
 A polynomial $p(x_1,\dotsc,x_n) \in \setR[x_1,\dotsc,x_n]$ is said 
 to be \defin{same-phase stable}  if for
 every $\lambda \in \setR_+^n$, we have that the univariate polynomial
 $p(\lambda_1 t,\lambda_2 t,\dotsc,\lambda_n t) \in \setR[t]$ is real-rooted.
\end{definition}

\begin{definition}
A weaker notion of interlacing is \emph{interleaving}
(we use the same terminology as in \cite{Branden2015}).
For $f,g \in \setR[t]$ with positive leading coefficients,
we say that $f$ is an \defin{interleaver} of $g$ if 
\[
\dotsb \leq f_2 \leq g_2 \leq f_1 \leq g_1,
\]
where $\{f_i\}_{i=1}^n$, $\{g_i\}_{i=1}^m$ are the zeros of $f$ and $g$,
respectively. We write this as $\defin{f \interl g }$.
Note that $f \preceq g \implies f \interl g$.

A sequence $\{f_1,f_2,\dotsc,f_n\}$ of polynomials
with non-negative coefficients,
is said to be an \defin{interlacing sequence}\footnote{Yes, we also think this is rather unfortunate terminology.}
if $f_i \interl f_j$  for all $1\leq i < j \leq n$.
We let $\defin{\intLacSec^+_n}$ denote the set of all such interlacing sequences.
\end{definition}
Let $E_n(\xvec)$ be the \defin{multivariate Eulerian polynomial}
\[
 E_n(\xvec) \coloneqq \sum_{\pi \in \symS_n}  \xvec_{\DES(\pi)}.
\]
In contrast with the $\tilde{E}_n(\xvec)$ we mentioned earlier,
the $E_n(\xvec)$ are not stable. For example
\begin{align*}
 E_5(\xvec) &= 
6 x_2 x_1 + 4 x_2 x_3 x_1 + 16 x_3x_1 + 9 x_2 x_4 x_1 \\
  &+x_2 x_3x_4 x_1 + 9 x_3 x_4 x_1 + 11 x_4 x_1 \\
  &+4 x_1+9 x_2+11 x_2x_3+9 x_3+16 x_2 x_4 \\
  &+4 x_2x_3 x_4+6 x_3 x_4+4 x_4+1
\end{align*}
vanishes at
\[
 x_1 = -\frac{39}{16}+\frac{7 i}{512}, \;
 x_2 = -16+i,\;
 x_3 = i,\;
 x_4 = \frac{-6523999+73341 i}{5671874}.
\]
However, we shall show that the $E_n(\xvec)$ 
are same-phase stable, and satisfy a type of interlacing.
\begin{theorem}\label{thm:eulierianMVinterlacing}
Let $\lambda_1,\lambda_2,\dotsc$ be a fixed sequence
of positive real numbers. Then for all $n\geq 1$,
\[
 E_{n-1}(\lambda_1 t,\lambda_2 t,\dotsc,\lambda_n t) \preceq 
 E_{n}(\lambda_1 t,\lambda_2 t,\dotsc,\lambda_n t).
\]
\end{theorem}
\begin{proof}
Let us first refine the polynomial $E_{n}(\xvec)$ by introducing
\[
 \defin{E^i_n(\xvec)} \coloneqq
 \sum_{\substack{\pi \in \symS_n \\ \pi(n)=i}}  \xvec_{\DES(\pi)}.
\]
We have that $E_{n-1}(\xvec) = E_n^n(\xvec)$,
since removing the last entry (which is $n$) in
a permutation counted by $E_n^n(\xvec)$, gives a bijection with 
elements counted by $E_{n-1}(\xvec)$.
In order to make expressions more compact later, we set
\begin{equation*}
   \defin{v^i_n(t)} \coloneqq E^i_n(\lambda_1 t, \lambda_2 t, \dotsc, \lambda_n t) , \qquad 
   \defin{v_n(t)} \coloneqq E_n(\lambda_1 t, \lambda_2 t, \dotsc, \lambda_n t).
\end{equation*}
We shall now prove 
that $\defin{\polySeqV_n} \coloneqq \{  v^i_n(t) \}_{i=1}^n$
is an interlacing sequence.
We follow the same approach as in \cite[Ex.~7.8.8]{Branden2015},
and note that by conditioning on $\pi(n-1)=k$, we have
\begin{equation}\label{eq:interlacingRecursion}
 v^i_{n+1}(t) = \sum_{k \geq i} \lambda_{n} t \cdot  v^k_{n}(t) + 
			    \sum_{k < i} v^k_{n}(t).
\end{equation}
Note that the weak inequality in the second sum is due to standardization.
For example, if $i=2$, and $\pi = 4356172$, then $\pi$
is obtained from the standardization of $\pi$ with the last element dropped, e.g., $324516$.

We now rephrase the recursion in \eqref{eq:interlacingRecursion} as
\begin{equation}
\begin{bmatrix}
v_{n+1}^{n+1} \\
v_{n+1}^n \\
v_{n+1}^{n-1} \\
\vdots \\
v_{n+1}^3 \\
v_{n+1}^2 \\
v_{n+1}^1 
\end{bmatrix}
=
\begin{bmatrix}
1 & 1 & 1 & \dotsc & 1 &  1 \\
\lambda_{n} t & 1 & 1 &  \dotsc & 1 &  1 \\
\lambda_{n} t & \lambda_{n} t & 1 & \dotsc & 1 &  1 \\
\lambda_{n} t & \lambda_{n} t & \lambda_{n} t & \dotsc & 1 &  1 \\
\vdots & \vdots &  & \ddots & \ddots & \vdots \\ 
\lambda_{n} t & \lambda_{n} t & \dotsc & \lambda_{n} t & \lambda_{n} t  &  1\\
\lambda_{n} t & \lambda_{n} t & \dotsc & \lambda_{n} t & \lambda_{n} t & \lambda_{n} t
\end{bmatrix}
\begin{bmatrix}
v_{n}^{n} \\
v_{n}^{n-1} \\
\vdots \\
v_{n}^3 \\
v_{n}^2 \\
v_{n}^1 
\end{bmatrix}
\end{equation}
where we denote the big matrix
by $\defin{G_n} \in \setR^{(n+1) \times n}$.
We now use \cite[Thm.~7.8.5]{Branden2015}, 
which allows us to easily verify that $G_n$
maps $\intLacSec^+_n$ to $\intLacSec^+_{n+1}$,
so (by induction) it follows that $\polySeqV_n$
is an interlacing sequence for all $n$
since the base case $\polySeqV_1 = (1)$ is trivially an 
interlacing sequence.

We now have that
\[
 v_{n-1}(t) = v^n_n(t) 
 \text{ and } v_n(t) = v^1_n(t) + \dotsb + v^n_n(t).
\]
Since $v_{n-1}(t)$ interleaves all
polynomials in $\{ v^i_n(t) \}_{i=1}^n$,
it must also interleave the sum, so $v_{n-1} \interl v_n(t)$.
Since $\deg(v_{n-1}) = \deg(v_n(t))$ or 
$\deg(v_{n-1})+1 = \deg(v_n(t))$,
we can now conclude that $v_{n-1} \preceq v_n(t)$,
and we are done.
\end{proof}

\begin{conjecture}
Let $A_n(\xvec)$ be as in \eqref{eq:rspDesRec1}.
Then $A_n(\xvec)$ is same-phase stable, 
and for all $n\geq 1$, we have that
\[
 A_{n-1}(\lambda_1 t,\lambda_2 t,\dotsc,\lambda_{n-1} t) \preceq
 A_{n}(\lambda_1 t,\lambda_2 t,\dotsc,\lambda_n t),
\]
whenever $\lambda_1,\lambda_2,\dotsc$ is a fixed sequence of positive
real numbers.
\end{conjecture}
Again, same-phase stability is the best we can hope for.
We have that $A_5(\xvec) = 1 + 3 x_1 + 5 x_2+3 x_3 + 3 x_1 x_3$,
but
$
A_5(-3+i,11/5+ 3i/5, i , i) = 0,
$
so $A_5(\xvec)$ has a zero in $\setH^4$, so $A_5(\xvec)$ is not stable.

As in the proof of \cref{thm:eulierianMVinterlacing},
a possible approach is to refine $A_{n}(\xvec)$ by introducing
\[
 A^i_n(\xvec) \coloneqq \sum_{\substack{\pi \in \RSP(n) \\ \pi(2)=i}}
 \prod_{j\in \DES(\pi)} \xvec_{n-j}.
\]
Now let $\lambda_1,\lambda_2,\dotsc$ be
positive real numbers and let $n$ be fixed.
Set
\[
   \defin{a^i_n(t)} \coloneqq
   A^i_n(\lambda_1t, \lambda_2 t, \dotsc, \lambda_n t).
\]
Computer experiments now suggest that $\{a^i_n(t)\}_{i=1}^n$ is an interlacing sequence. Refining based on the last entry of $\pi$
as we did in \cref{thm:eulierianMVinterlacing}
produces the polynomials
\[
 \sum_{\substack{\pi \in \RSP(n) \\ \pi(n)=i}}
 \prod_{j\in \DES(\pi)} \lambda_{n-j} t, \qquad 2 \leq i \leq n,
\]
which also seem to be real-rooted.
However, we do not have interlacing in this situation.

\section{Descents in a permutation after run-sort}\label{C}

We study the distribution
of descents after run-sorting a permutation.
In \cref{ssec:peaks}, we first give some historical background and results concerning peaks in permutations.
We then prove the main technical results in \cref{ssec:peaksAfterSort},
which amount to a type of bijective argument.
In \cref{ssec:conclusion}, we use the previous arguments to prove the main result of this article,
and in the final subsection we study some consequences of our result.

\subsection{Peaks in permutations}\label{ssec:peaks}

We let $\defin{\bh_{n,k}}$ be the number of permutations in $\symS_n$ with exactly $k$ peaks.
There is a nice recursion for the numbers $\bh_{n,k}$ given first in \cite[p.24]{Liagre1855},
and has been discovered numerous times since then.
We have that
\begin{equation}\label{eq:rr1}
\bh_{n,k} = 
\begin{cases}
1  & \text{ if $n=1$, $k=0$}, \\
(2k + 2)\bh_{n - 1, k} + (n- 2(k-1)-2) \bh_{n - 1, k - 1} & \text{ if $0 \leq k < \frac{n}{2}$}, \\
0 & \text{ if $2k \geq n$ or $k<0$}.
\end{cases}
\end{equation}
These numbers can be found in \oeis{A008303}.
Now, set
\[
\defin{\Bh_n(t)} \coloneqq \sum_{j \geq 0} \bh_{n,j} t^j, 
\qquad \defin{G(u,t)} \coloneqq \sum_{n\geq 0} \frac{\Bh_n(t)}{n!} u^n.
\]
In \cite{WarrenSeneta1996}, it is proved that 
\begin{equation}\label{eq:numPeakRecursion}
 \Bh_n(t) = (2+t(n-2)) \Bh_{n-1}(t) + 2t(1-t) \Bh'_{n-1}(t), \text{ for $n \geq 2$}.
\end{equation}
We prove a multivariate generalization of this recursion further down in \cref{prop:mvPVSRec}.
D.~Warren and E.~Seneta also proved that all zeros of $\Bh_n(t)$ are real and that $\Bh_{n-1} \preceq \Bh(n)$
by using \eqref{eq:numPeakRecursion}.
We believe that a multivariate generalization of this holds,
see~\cref{conj:peakValuePolynomialStable} below.

In \cite[Cor. 23]{Kitaev2007} it is shown that
 \begin{align}\label{eq:peakGenFunc}
 G(u, t) &=  - \frac{1}{t} + \frac{\sqrt{t - 1}}{t} \tan(u \sqrt{t - 1} + \arctan(1/\sqrt{t - 1})), \\
          &= \dfrac{\tan(u\sqrt{t-1})}{\sqrt{t-1} - \tan(u\sqrt{t-1})},
\end{align}
where the second follows from the first by using Mathematica.

The first few values of $\Bh_n(t)$ are presented in \cref{tab:peakPolys}.
\begin{table}[!ht]
\begin{tabular}{cr}
$n$  & $\Bh_n(t)$\\
\toprule
 1 & $1$ \\
 2 & $2$ \\
 3 & $2 t+4$ \\
 4 & $16 t+8$ \\
 5 & $16 t^2+88 t+16$ \\
 6 & $272 t^2+416 t+32$ \\
\bottomrule
\end{tabular}
\caption{The polynomials $\Bh_n(t)$.
For example, there are $272$ permutations in $\symS_6$ with $2$ peaks.}\label{tab:peakPolys}
\end{table}

For the remainder of this subsection, we shall focus our attention 
on the multivariate polynomial
\begin{equation}\label{eq:mvPeakPoly}
 \defin{\Bh_n(\xvec)} \coloneqq \sum_{\pi \in \symS_n} \xvec_{\PEAKVAL(\pi)},
\end{equation}
which generalizes $\Bh_n(t)$. 
Note that $\Bh_n(\xvec)$ is \defin{multiaffine} i.e., it is linear in each fixed variable.
We shall use the shorthand $\defin{\partial_{x_j}} \coloneqq \partial/\partial_{x_j}$.

For $\pi \in \symS_{n-1}$ and $a\in[n-1]$, we let 
$\defin{\insStay_a(\pi)} $
denote the permutation obtained from $\pi$ by inserting $n$ immediately after $a$.
We also let $\insStay_\emptyset(\pi)$ denote the permutation obtained from $\pi$
by inserting $n$ before $\pi$ in one-line notation.

The following lemma allows us to track peaks when recursively constructing permutations.
The motivation for the two sub-cases in (5) will be clear later in 
\cref{lem:otherStatistics}. For now, all arguments below do not need
to distinguish the two sub-cases, unless stated explicitly.
\begin{lemma}\label{lem:insertNforPeaks}
For any $n\geq 1$, the bijection 
\[
  \peakBij : \{\emptyset,1,2,\dotsc,n-1 \} \times \symS_{n-1} \to  \symS_{n}
\]
defined via $\defin{\peakBij(a,\pi)} \coloneqq \insStay_a(\pi)$, has the following properties.
For simplicity, we set $\pi' \coloneqq \insStay_a(\pi)$ and 
we let $k$ be the value immediately succeeding $a$ in $\pi$ (unless $a$ is the last entry in $\pi$).
\begin{enumerate}
 \item $a=\emptyset$, so $\PEAKVAL(\pi') = \PEAKVAL(\pi)$.
 
 \item $a$ is the last entry of $\pi$, so $\PEAKVAL(\pi') = \PEAKVAL(\pi)$.
 
 \item \label{en:insertNafterPeak} $a \in \PEAKVAL(\pi)$. Then 
 \[
  \PEAKVAL(\pi') = (\PEAKVAL(\pi) \setminus \{a\}) \cup \{n\}.
 \]
 \item \label{en:insertNbeforePeak}  $k \in \PEAKVAL(\pi)$. Then 
 \[
  \PEAKVAL(\pi') = (\PEAKVAL(\pi) \setminus \{k\}) \cup \{n\}.
 \]
 
 \item $a$ is not the last entry of $\pi$, and neither $a$ or $k$ are in $\PEAKVAL(\pi)$. 
 Then 
 \begin{enumerate}
 \item if $a<k$
 \[
  \PEAKVAL(\pi') = \PEAKVAL(\pi)  \cup \{n\},
 \]
 \item otherwise $a>k$, and we also have
 \[
  \PEAKVAL(\pi') = \PEAKVAL(\pi)  \cup \{n\}.
 \]
\end{enumerate}
\end{enumerate}
\end{lemma}
\begin{proof}
First note that the map $\peakBij$ is indeed a bijection, as we can easily recover $a$
from $\pi'$, and $\pi$ is recovered from $\pi'$ by removing $n$.

The other properties regarding the peaks follow via case-by-case analysis.
We should perhaps clarify why the cases (\ref{en:insertNafterPeak}) and (\ref{en:insertNbeforePeak}) 
are indeed mutually exclusive; two elements in $\PEAKVAL(\pi)$ cannot be adjacent in $\pi$.
\end{proof}

\begin{lemma}\label{lem:casesCount}
Let $\pi \in \symS_{n-1}$ and $m=|\PEAKVAL(\pi)|$ and $r=|\DB(\pi)|$.
We can then count how many values of $a \in \{\emptyset,1,2,\dotsc,n-1\}$,
belong to each case in \cref{lem:insertNforPeaks}.
\begin{enumerate}
 \item For Case 1, there is only one option;
 \item For Case 2, there is only one option;
 \item For Case 3, there are $m$ options;
 \item For Case 4, there are $m$ options;
 \item For Case 5a, there are $(n-2)-r-m$ options;
 \item For Case 5b, there are $r-m$ options.
 \end{enumerate}
\end{lemma}
\begin{proof}
Cases 1--4 follow directly from the the definition of the cases, 
e.g., we can insert $n$ after each peak value, so there are $m$ options in Case~3.

Moreover, if we add all options listed, we get $n$ (as expected),
so it suffices to show that the number of options in Case 5a is $(n-2)-r-m$. 

Consider a permutation $\pi \in \symS_{n-1}$.
Then there are a total of $n-2$ spots where $n$ can be inserted between two entries.
However, there are $m = |\PEAKVAL(\pi)|$ forbidden spots,
as we are not allowed to insert $n$ before a peak. 
Moreover, we are not allowed to insert $n$ between $a$ and $k$,
if $a>k$, as this is Case 5b. There are exactly $r$
such occurrences of $a$.
Hence, what remains are $(n-2)-r-m$ valid 
options of $a$, where $a < k$.
\end{proof}

As a corollary of \cref{lem:insertNforPeaks}, we can now easily deduce \eqref{eq:rr1},
and prove the following multivariate generalization of \eqref{eq:numPeakRecursion}.
\begin{proposition}\label{prop:mvPVSRec}
We have that the multivariate polynomials $\Bh_{n}(\xvec)$ satisfy the recursion 
  \begin{equation}\label{eq:mvPVSRec}
  \Bh_{n} = (2+ (n-2) x_n ) \Bh_{n-1} + 
  2 x_n\sum_{j=3}^n (1 - x_j) \cdot \partial_{x_j} \Bh_{n-1} .
 \end{equation}
\end{proposition}
\begin{proof}
We shall prove that the recursion in \eqref{eq:mvPVSRec}
has the same structure as the recursion described in \cref{lem:insertNforPeaks}.
We rewrite \eqref{eq:mvPVSRec} as
\begin{align*}
  \Bh_{n} &= \Bh_{n-1}  \tag{Case 1}\\
          & +\Bh_{n-1}  \tag{Case 2}\\
          & +  x_n \sum_{a=3}^n \partial_{x_a}  \Bh_{n-1}  \tag{Case 3}\\
          & +  x_n \sum_{k=3}^n \partial_{x_k}  \Bh_{n-1}  \tag{Case 4}\\
          & +  x_n \left[  (n-2)\Bh_{n-1} - x_a \left(\sum_{a=3}^n \partial_{x_a}  \Bh_{n-1} \right)
          - x_k \left(\sum_{k=3}^n \partial_{x_k}  \Bh_{n-1} \right) \right]
          \tag{Case 5}.
\end{align*}
It remains to show that the equation labels indeed correspond with the five cases in \cref{lem:insertNforPeaks}.
That $\Bh_{n-1}$ corresponds to cases~1~and~2 is evident.

In Case~3, we use $\partial_{x_a}$ to 
select all permutations counted by $\Bh_{n-1}$ where $a$ is a peak.
This peak is then replaced by inserting $n$ after $a$, which makes $n$ a new peak.
A similar argument applies in Case~4.

In the last case, for any permutation of length $n-1$,
there are $n-2$ potential spots where we may insert $n$ to make a peak,
and a new permutation $\pi \in \symS_n$.
However, we do not want to include the cases where a previous peak 
is replaced as these are covered by Case~3 and ~4.
We therefore want to exclude the permutations counted 
by $\Bh_{n-1}$, where the entry immediately before $n$ in $\pi$ was a peak,
or the entry immediately after $n$ was a peak.
\end{proof}

As mentioned earlier, it has been shown that the $\Bh_n(t)$
are real-rooted. Computer experiments suggest
the following stronger property for the multivariate version.
\begin{conjecture}\label{conj:peakValuePolynomialStable}
The polynomials $\Bh_n(\xvec) = \sum_{\pi \in \symS_n} \xvec_{\PEAKVAL(\pi)}$,
are all stable.
Moreover, of $\lambda_1,\lambda_2,\dotsc$ is a fixed sequence of positive real numbers,
then 
\[
 \Bh_{n-1}(\lambda_1 t,\lambda_2 t,\dotsc,\lambda_{n-1} t) \preceq
 \Bh_{n}(\lambda_1 t,\lambda_2 t,\dotsc,\lambda_n t) \text{ for all $n>1$}.
\]
\end{conjecture}

A family of multi-affine polynomials of similar flavor 
is considered in \cite[p.10]{BrandenLeander2016x},
which are shown to be \emph{Hurwitz stable}.

\subsection{A recursion which tracks peak-values after run-sort}\label{ssec:peaksAfterSort}

We have a recursive structure of permutations and peak-values
described in \cref{lem:insertNforPeaks}.
Our goal is now to describe an analogous recursive
structure for permutations, but now keeping track
of the set of peak-values after run-sorting.
Again we keep track of the position where $n$ is inserted,
and have five cases corresponding to the cases in \cref{lem:insertNforPeaks}.
However, some operations are more complicated, and we employ some operations 
on runs which are reminiscent of crossover on strands of DNA.
A rough overview of our approach is given in \cref{fig:fiveCasesBijFig}. 
\begin{figure}[!ht]
\centering
\includegraphics[width=0.6\textwidth,page=1]{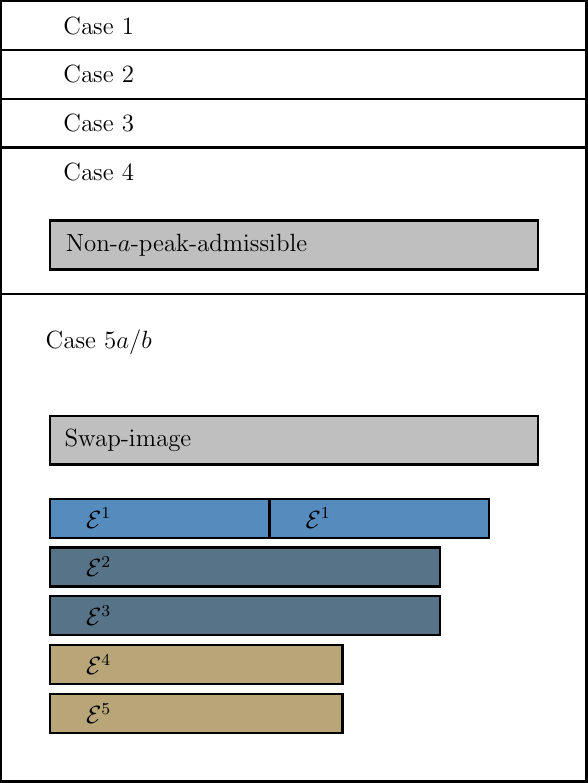} 
\caption{
Cases $1$, $2$ and $3$ are easy to handle, i.e., by simply using $\insStay_a$ on the permutation.
In the remaining cases, we use $\insStay_a$ whenever it has the desired effect.
Otherwise, we do some modification. This modification has to be compensated for later.
This leads to a handful of involutions which cancel the cases not handled by $\insStay_a$.
In the figure, two shaded sets of the same size and color are put in bijective correspondence via such an involution.
The notation and terminology is explained later in this section.
}\label{fig:fiveCasesBijFig}
\end{figure}

Let $\defin{\SPV(\pi)}\coloneqq \PEAKVAL(\runsort(\pi))$
be the set of peak-values obtained after run-sorting $\pi$.
In this subsection, we shall mainly consider the runs of $\pi \in \symS_{n-1}$ 
arranged in lexicographical order.
We illustrate the runs of $\pi$ as shown in \eqref{eq:lexsortPiRunsIntro},
where the $s_i$ are the start of the runs, with $s_1 < s_2 < \dotsb < s_r$:
\begin{equation}\label{eq:lexsortPiRunsIntro}
 \runBlock{s_1}{}{e_1} \;
 \runBlock{s_2}{}{e_2} \;
 \dotsc  \;
 \runBlock{s_r}{}{e_r}.
\end{equation}
By definition, $\SPV(\pi) \subseteq \{ e_1, e_2, \dotsc,e_{r-1} \}$.
\emph{Note that it is not guaranteed that $e_i > s_{i+1}$ for all $i\in[r-1]$, since 
\eqref{eq:lexsortPiRunsIntro} lists the runs of $\pi$, not the runs of $\runsort(\pi)$.
}
In this context, when referring to the runs of $\pi$, the words \emph{last}, \emph{next}, etc.,
are with respect to the ordering in \eqref{eq:lexsortPiRunsIntro}, and not the 
order in which they appear in $\pi$.

From now on, we have a fixed $\defin{\pi} \in \symS_{n-1}$ and $\defin{a} \in \{1,2,\dotsc,n-1\}$
present in the context we are working in.
Let $\defin{\alpha}$ denote the run of $\pi$
containing $a$ and let $\defin{k}$ be the value immediately succeeding $a$, if such 
an entry exist.
Furthermore, $\defin{\beta}$ is the run succeeding $\alpha$ (in lex-order)
and $\defin{\omega}$ is the last run of $\pi$.
Now, either $a$ and $k$ are both in $\alpha$, or $a$ is the largest value in $\alpha$.
In the latter case, either $\alpha$ is not the last run, and $k$ is the smallest entry in $\beta$,
or $a$ is simply the very last entry in $\runsort(\pi)$, see \cref{eq:lexsortPiRunsCases1,eq:lexsortPiRunsCases2,eq:lexsortPiRunsCases3}.

\begin{align}
 \dotsc
 \underset{\alpha}{ \runBlock{\cdots}{ak}{\cdots} } \;
 \underset{\beta}{ \runBlock{}{}{} } \;
 \dotsc  \;
 \underset{\omega}{\runBlock{s_r}{}{e_r}} \label{eq:lexsortPiRunsCases1} \\
 \dotsc
 \underset{\alpha}{ \runBlock{}{}{a} } \;
 \underset{\beta}{ \runBlock{k}{}{} } \;
 \dotsc  \;
 \underset{\omega}{\runBlock{s_r}{}{e_r}} \label{eq:lexsortPiRunsCases2}  \\
 \dotsc
 \dotsc  \;
 \underset{\omega}{\runBlock{s_r}{}{a}}  \label{eq:lexsortPiRunsCases3}
\end{align}

\medskip 

Suppose now $a$ is the penultimate entry
in $\alpha$ and thus $k$ is the largest entry in $\alpha$:
\[
 \dotsc
 \underset{\alpha}{ \runBlock{}{}{ak} } \;
 \underset{\beta}{ \runBlock{}{}{} } \;
 \dotsc 
\]
In the case $k \in \SPV(\pi)$, we say
that $\pi$ is \defin{$a$-peak-admissible} if
\begin{equation}\label{eq:peakAdmissible}
 \SPV(\insStay_a(\pi)) = (\SPV(\pi) \setminus \{k\}) \cup \{n\},
\end{equation}
that is, inserting $n$ after $a$, replaces $k$ by $n$ as a peak value after run-sorting.
(Recall that $\insStay_a(\pi)$ was defined as the permutation obtained
from $\pi$ by inserting $n$ immediately after $a$.)

\begin{example}
For example, consider
\[
 \sigma = 52674318, \text{ and } \pi = 38256714 \text{ which lie in $\symS_8$.}
\]
We have that $\SPV(\sigma) = \{7,8\}$ and $\SPV(\pi) = \{4,7\}$,
$\insStay_6(\sigma) = 526974318$ and $\insStay_6(\pi) = 382569714$,
$\SPV(\insStay_6(\sigma)) = \{8, 9\} = (\SPV(\sigma) \setminus \{7\}) \cup \{9\}$,
$\SPV(\insStay_6(\pi)) = \{4, 8, 9\} \neq (\SPV(\pi) \setminus \{7\}) \cup \{9\}$.

Hence, $\sigma$ is $6$-peak-admissible, but $\pi$ is not.
\end{example}

The following lemma characterizes the $a$-peak-admissible permutations.
\begin{lemma}\label{lem:kAdmissChar}
Let $\pi \in \symS_{n-1}$ be a permutation with $k \in \SPV(\pi)$,
and let $a$ be the entry immediately before $k$:
\begin{equation}
\runBlock{s_1}{}{e_1}
\dotsc
 \underset{\alpha}{ \runBlock{}{}{ak} } \;
 \underset{\beta}{ \runBlock{b}{}{} } \;
 \dotsc  \;
 \underset{\gamma}{\runBlock{s_m}{}{e_m}} 
 \dotsc  \;
 \runBlock{s_r}{}{e_r}.
\end{equation}
Let $m \in [r]$ be the largest value such that $s_m<k$.
Then $\pi$ is $a$-peak-admissible if and only if
\begin{enumerate}[label=(\roman*)]
 \item $e_m<k$, or
 \item $m<r$ and $e_m > s_{m+1}$.
\end{enumerate}
Moreover, if $\pi$ is \emph{not} $a$-peak-admissible, 
then $s_m<k<e_m < s_{m+1}$ and 
\begin{equation}\label{eq:non-a-peak-case}
 \SPV(\insStay_a(\pi)) =  (\SPV(\pi) \setminus \{k\}) \cup \{e_m,n\}.
\end{equation}
\end{lemma}
\begin{proof}
First, it is clear that the number $m$ exists; since $k \in \SPV(\pi)$,
we know that $b<k$, so the set of runs where the first entry is less than $k$
is non-empty. 

Note that by the choice of $m$, we have additional inequalities in the two cases:
\begin{enumerate}[label=(\roman*)]
 \item $e_m < k < s_{m+1}$, or
 \item $m<r$ and $k < s_{m+1} < e_m$.
\end{enumerate}
Suppose $\pi$ is $a$-peak-admissible and let $\pi' = \insStay_a(\pi)$.

Consider the lex-largest run of $\pi'$, where the starting point $s_m < k$.
We have two cases to consider here i.e., whether $m=r$ or $m<r$.
(i) If $m=r$, then $k$ appears at the end of $\pi'$ and
hence will not be an element in $\SPV(\pi')$.
Otherwise if $r<m$ and $k < s_{m+1}$, then $k$ appears immediately after $\gamma$ in $\pi'$
and before $s_{m+1}$ in $\pi'$. Hence, $k$ is again not a peak-value of $\runsort(\pi')$.
Moreover, $e_m \notin \SPV(\pi)$ and we still have $e_m \notin \SPV(\pi')$.

(ii) Since $m < r$, we know that $k < s_{m+1} $. Moreover, since $e_m > s_{m+1}$
we have that $e_m \in \SPV(\pi)$. 
Again, $k$ appears immediately after $e_m$ in $\runsort(\pi')$
and we still have $e_m \in \SPV(\pi')$.
Hence, we have now shown that $\pi$ is $a$-peak-admissible in both cases listed.

\medskip 
It is now straightforward to verify that the condition $s_m <k < e_m < s_{m+1}$
is precisely the complement of the events (i) and (ii) above.
Thus, we only need to show that $\SPV(\insStay_a(\pi)) =  (\SPV(\pi) \setminus \{k\}) \cup \{e_m,n\}$,
whenever $s_m <k < e_m < s_{m+1}$. 
But from these inequalities, we can deduce
that $e_m \notin \SPV(\pi)$ and that the runs of $\insStay_a(\pi)$
are arranged lexicographically as
\begin{equation}
\runBlock{s_1}{}{e_1}
\dotsc
 \underset{\alpha}{ \runBlock{}{}{an} } \;
 \underset{\beta}{ \runBlock{b}{}{} } \;
 \dotsc  \;
 \underset{\gamma}{\runBlock{s_m}{}{e_m}}
 \;
 \singleBlock{k}
 \runBlock{s_{m+1}}{}{e_{m+1}}
\end{equation}
It is now evident that $e_m \in \SPV(\insStay_a(\pi))$, 
and that \eqref{eq:non-a-peak-case} is true.
\end{proof}

\bigskip 

\begin{definition}

Suppose now that $\pi \in \symS_{n-1}$, $k \in \SPV(\pi)$
and $a$ preceding $k$, but $\pi$ is \emph{not} $a$-peak-admissible.
We then construct a permutation denoted $\insSwap_a(\pi)$ via the following procedure.
\begin{enumerate}[label=({\roman*})]
\item Let
\begin{equation}\label{eq:runsOfPi}
\dotsc
 \underset{\alpha}{ \runBlock{}{}{ak} } \;
 \underset{\beta}{ \runBlock{b}{}{} } \;
 \dotsc  \;
 \underset{\gamma}{\runBlock{s_m}{}{e_m}} 
 \dotsc  \;
\end{equation}
be the runs of $\pi$, as in \cref{lem:kAdmissChar}.

\item Let $\gamma$ be the largest run with $s_m<k<e_m$ 
(there is such a run, according to \cref{lem:kAdmissChar}).

\item We split $\gamma$ as $\gamma_1 \gamma_2$, 
where $\gamma_1$ is the prefix of $\gamma$
consisting of the elements which are less than $k$,
while $\gamma_2$ are the remaining elements.
\[
 \underset{\gamma}{\runBlock{s_m}{}{e_m}}  = 
 \underset{\gamma_1}{\shortBlock{s_m}{}}
 \underset{\gamma_2}{\shortBlock{}{e_m}}
\]
We let $\defin{\insSwap_a(\pi)}$ be the permutation obtained from $\pi$
by moving the elements in $\gamma_2$, and inserting them immediately after $k$.
Hence, the lex-sorted runs of $\insSwap_a(\pi)$ are 
\begin{equation}\label{eq:insSwapDef}
\dotsc
 \underset{\alpha}{ \runBlock{}{}{ak} }
%  \singleBlock{k}
 \underset{\gamma_2}{\shortBlock{}{e_m}} \;
 \;
 \underset{\beta}{ \runBlock{b}{}{} } \;
 \dotsc  \;
 \underset{\gamma_1}{ \shortBlock{s_m}{} }
 \;
 \runBlock{s_{m+1}}{}{e_{m+1}}
 \dotsc  \;
\end{equation}

\end{enumerate}
\end{definition}

The run $\gamma$ does not need to be located somewhere after $\alpha$ in $\pi$,
it can be located anywhere in $\pi$.
The run immediately after $\gamma_1$
in $\insSwap_a(\pi)$, (let's call it $\delta$) must start with something 
less than $e_m$ since otherwise, $e_m$ would not be the largest 
value in its run. Moreover, since $s_{m+1}$ is the smallest starting 
value of a run greater than $k$, and $s_{m+1}>e_m$ (because $\pi$ is non-$a$-peak-admissible),
it follows that the first entry of $\delta$ must be less than $s_m$.
This implies that $\gamma_1$ does not `merge' with $\delta$,
as $\gamma_2$ is moved to a different part of $\pi$.
In other words, $\gamma_1$ is still a run in $\insSwap_a(\pi)$.
The runs of $\insStay_a(\insSwap_a(\pi))$ in lexicographical order are therefore
 \begin{equation}\label{eq:insSwapLexSorted}
\dotsc
 \underset{\alpha}{ \runBlock{}{}{an} } \;
 \underset{\beta}{ \runBlock{b}{}{} } \;
 \dotsc  \;
 \underset{\gamma_1}{ \shortBlock{s_m}{} }
 \singleBlock{k}
 \underset{\gamma_2}{\shortBlock{}{e_m}} \;
 \runBlock{s_{m+1}}{}{e_{m+1}}
 \dotsc  \;.
\end{equation}

\begin{example}
We use the same notation as in the description of $\insSwap_a$.
Let $\pi = 38256714$ and $a=6$, $k=7$.
The runs of $\pi$ in lex-order are
\[
  14 \; 2567\; 38.
\]
Then $\gamma = 38$ and $\gamma_1 = 3$, $\gamma_2 = 8$.
Hence, $\insSwap_a(\pi) = 32567814$ because $8$ is moved and inserted after $k=7$.
\end{example}

\begin{lemma}\label{lem:swap}
Let $D_n$ be the subset of pairs $(a,\pi) \in  [n-1] \times \symS_{n-1}$
such that $a$ immediately precedes an entry in $\SPV(\pi)$,
but $\pi$ is not $a$-peak-admissible.

Then the map from $D_n$ to $\symS_n$ defined by 
\begin{equation}\label{eq:swapStayMap}
 (a, \pi) \mapsto \insStay_a(\insSwap_a(\pi))
\end{equation}
is well-defined and injective.
Moreover, for $(a,\pi) \in D_n$, we have that
\begin{equation}\label{eq:peakNonAdmissible}
 \SPV(\insStay_a(\insSwap_a(\pi))) = (\SPV(\pi) \setminus \{k\}) \cup \{n\}.
\end{equation}
\end{lemma}
\begin{proof}
Since $\insSwap_a(\pi)$ is well-defined for $(a,\pi)\in D_n$,
it follows that the map in \eqref{eq:swapStayMap} is also well-defined. 
Now, given the output of $\insStay_a$, we can recover $a$
(it is the entry immediately to the left of $n$) and thus $k$.
It now suffices to show that $\insSwap_a$ is injective on the set of permutations 
which are not $a$-peak-admissible. However, given the output of $\insSwap_a(\pi)$
for some known $a$, we can identify $k$ and then also $\gamma_2$ for $\pi$.
Now $\gamma_1$ is the lex-largest run whose first entry is less than $k$.
It is now clear that we can recover $\pi$ from $\insStay_a(\insSwap_a(\pi))$.
The fact that \eqref{eq:peakNonAdmissible} holds follows from comparing 
\eqref{eq:runsOfPi} with \eqref{eq:insSwapLexSorted}
---it is evident that $n$ is a peak-value, and that $k$ is now a starting point of a run in $\pi$ (so in particular, not a peak-value). All other peak-values after run-sort are preserved.
\end{proof}

\bigskip 

\begin{definition}[Swap-image]
If there is some non-$a$-peak-admissible 
$\sigma \in \symS_{n-1}$ such that $\sigma = \insSwap_a(\pi)$,
we say that $\pi \in \symS_{n-1}$ is in the \defin{swap-image}.
\end{definition}
The permutations which are not in the swap-image,
are handled later in this section.

\bigskip 

Next, we treat the analog of the last
case in \cref{lem:insertNforPeaks},
where we insert $n$ in such a way that it creates a new peak.
For $\pi \in \symS_{n-1}$, let $\defin{\slopeSet(\pi)}$ be the values of $a \in [n-1]$
such that if we insert $n$ after $a$ in $\runsort(\pi)$, 
$n$ becomes a peak-value and all other peak values are
preserved\footnote{Note: Inserting $n$ does not preserve being run-sorted in general.}.
If we let $(\pi'_1, \pi'_2,\dotsc,\pi'_{n-1}) \coloneqq \runsort(\pi)$,
then $\slopeSet(\pi)$ can be
described more explicitly as
\begin{align*}
\slopeSet(\pi) &= 
\left\{ a : \SPV(\insStay_a(\pi)) = \SPV(\pi) \cup \{n\} \right\} \\
&= \left\{ \pi'_i :
 \begin{array}{ll}
   \text{$i \leq n-2 $, $\pi'_i \notin \SPV(\pi)$, and $\pi'_{i+1} \notin \SPV(\pi)$}
   \end{array}
   \right\}. 
\end{align*}

\begin{example}\label{ex:kslopeAdmissible}
Some permutations and their sets $\slopeSet(\pi)$
are shown below.
\begin{center}
\begin{tabular}{lll}
 Permutation $\pi$ & $\runsort(\pi)$ & $\slopeSet(\pi)$ \\
 \toprule 
 $2561734$ &  $1725634$ & $\{2,3\}$ \\
 $4567123$ &  $1234567$ & $\{1,2,3,4,5,6\}$ \\
 $4 3 7 1 6 2 5$ &  $1625374$ & $\{  \}$ \\
 \bottomrule
\end{tabular}
\end{center}
\end{example}

We say that $\pi$ is \defin{$a$-slope-admissible} if 
\begin{equation}\label{eq:slopeAdmissible}
  a \in \slopeSet(\pi) \text{ and } \SPV(\insStay_a(\pi)) = \SPV(\pi) \cup \{n\}.
\end{equation}

The next lemma contains additional technical results
needed later for the main conclusion.
We characterize the $a$-slope-admissible permutations
and show that these are never in the swap-image.
We also describe all non-$a$-slope-admissible permutations
which are not in the swap-image.
\begin{lemma}\label{lem:aAdmissChar}
Let $\pi \in \symS_{n-1}$ be a permutation with $a \in \slopeSet(\pi)$
and write $\alpha$  as $\defin{\alpha_1}\; a \; \defin{\alpha_2}$:
\begin{equation}
\dotsc
 \underset{\alpha}{ \runBlock{}{a}{} } \;
 \underset{\beta}{ \runBlock{b}{}{} } \;
 \dotsc  \;
 \underset{\gamma}{\runBlock{s_{m+1}}{}{e_{m+1}}}
 \dotsc  \;.
\end{equation}
In the case $\alpha_2 \neq \emptyset$, let \defin{$k$} and \defin{$l$} be the smallest and largest
values in $\alpha_2$, respectively.
Moreover, let $\defin{m} \in [r-1]$ be the smallest value where $s_{m+1}>k$, if such an $m$ exists.
With all these definitions, the runs of $\pi$ are as follows:
\begin{equation}\label{eq:aSlopeGeneric}
\dotsc
\underset{\alpha_1}{\shortBlock{}{}}
 \singleBlock{a}
 \underset{\alpha_2}{\shortBlock{k}{l}} \;
 \underset{\beta}{ \runBlock{b}{}{} } \;
 \dotsc  \;
 \runBlock{s_{m+1}}{}{e_{m+1}} 
 \dotsc. 
\end{equation}
\medskip 
\noindent
\textbf{Statement 1}: We have that $\pi$ is $a$-slope-admissible if and only if
\begin{enumerate}[label=(\roman*)]
 \item $\alpha_2 = \emptyset$, or
 \item $\alpha$ is the lex-largest run, or
 \item $k<b$, or
 \item $m$ exists, $l > s_{m+1}$
 and 
 $
  e_{m} < k \iff e_{m} < s_{m+1}.
 $
\end{enumerate}

\medskip 
\noindent
\textbf{Statement 2}: We have that $\pi$ is \emph{not} $a$-slope-admissible if and only if
\begin{enumerate}[label=(\roman*')]
 \item $\alpha_2$ is non-empty,
 \item $\alpha$ is not the last run (so $b$ exists),
 \item $b<k$ and
 \item either 
 \[
   m \text{ does not exist, or } 
   l < s_{m+1} \text{ or }
   k < e_{m} < s_{m+1}.
 \]
\end{enumerate}

\medskip 
\noindent
\textbf{Statement 3}: If $\pi$ is in the swap-image,
then $\pi$ is not $a$-peak-admissible.

\medskip 
\noindent
\textbf{Statement 4}: if $\pi$ is not $a$-peak-admissible
and not in the swap-image, then one of the following conditions hold:
\begin{enumerate} %[label=(\roman*)]
 \item \label{it:ee1}
 $m$ does not exist, (so $s_j<k$ for all $j\in[r]$),
 $e_r>k$ and
 \[
 \SPV(\insStay_a(\pi)) = (\SPV(\pi)\setminus\{l\})\cup \{e_r,n\}.
\]

\item \label{it:ee2}
$k< l < e_m < s_{m+1}$, and
\[
 \SPV(\insStay_a(\pi)) = (\SPV(\pi)\setminus\{l\})\cup \{e_m,n\}.
\]

\item  \label{it:ee3}
$k< e_m <l< s_{m+1}$, and
\[
 \SPV(\insStay_a(\pi)) = (\SPV(\pi)\setminus\{l\})\cup \{e_m,n\}.
\]

\item  \label{it:ee4}
$k< l < s_{m+1} < e_m$, and
\[
 \SPV(\insStay_a(\pi)) = (\SPV(\pi)\setminus\{l\})\cup \{n\}.
\]

\item  \label{it:ee5}
$k<e_m<s_{m+1}<l$, and
 \[
\SPV(\insStay_a(\pi)) = \SPV(\pi)\cup \{e_m,n\}.
\]
\end{enumerate}
\end{lemma}
\begin{proof}
\noindent
\textbf{Proof of statement 1}: 
We first verify that cases (i)--(iv) are indeed $a$-slope-admissible.
The first three cases are easy, since inserting $n$ after $a$
splits $\alpha$ into two smaller runs, but the two resulting runs are
adjacent in the lex-order, as in \eqref{eq:aSlopeGeneric}.
Thus, only case (iv) requires some more work.

So first we suppose that $l > s_{m+1}$ and that both $e_{m} < k$ and $e_{m} < s_{m+1}$
hold, i.e., $e_m$ is not a peak-value.
Then, inserting $n$ after $a$ in \eqref{eq:aSlopeGeneric}, 
gives that $\insStay_a(\pi)$ has the runs
\begin{equation}\label{eq:aSlopeGenericAfterAStay}
\dotsc
\underset{\alpha_1}{\shortBlock{}{an}} \;
 \underset{\beta}{ \runBlock{b}{}{} } \;
 \dotsc  \;
 \runBlock{s_{m}}{}{e_{m}} \;
  \underset{\alpha_2}{\shortBlock{k}{l}} \;
 \runBlock{s_{m+1}}{}{e_{m+1}}
 \dotsc 
\end{equation}
Since $l > s_{m+1}$, we still have that $l$ is a peak-value,
and $e_m<k$ implies that $e_m$ is still not a peak-value.
Hence, $\pi$ is $a$-peak-admissible.
The second sub-case, when $l > s_{m+1}$ but both $e_{m} < k$ and $e_{m} < s_{m+1}$ are false,
is treated in a similar manner.

\medskip 
\noindent
\textbf{Proof of statement 2}: 
We now show that $\pi$ is \emph{not} $a$-slope-admissible if
$\alpha_2 \neq \emptyset$, $b<k$ and one of the cases
\[
\text{(a)}\quad \text{$m$ does not exist}, \qquad
\text{(b)}\quad l < s_{m+1}, \qquad
\text{(c)}\quad  k < e_m < s_{m+1}
\]
hold.

\medskip 
\noindent
\textbf{In case (a)}, we have that $s_j<k$ for all $j \in [r]$, so the runs of $\insStay_a(\pi)$
in lex-order are
\begin{equation}\label{eq:casea}
\dotsc
\underset{\alpha_1}{\shortBlock{}{an}} \;
 \underset{\beta}{ \runBlock{b}{}{} } \;
 \dotsc  \;
 \runBlock{s_{r}}{}{e_{r}} \;
 \underset{\alpha_2}{\shortBlock{k}{l}}.
\end{equation}
Hence, 
\begin{align}
\SPV(\insStay_a(\pi)) = 
\begin{cases}
 (\SPV(\pi)\setminus\{l\})\cup \{n\} & \text{ if } e_r<k  \label{eq:casea1}\\
 (\SPV(\pi)\setminus\{l\})\cup \{e_r,n\}& \text{ if } e_r>k,
\end{cases}
\end{align}
and neither case is therefore $a$-peak-admissible.

\medskip 
\noindent
\textbf{In cases (b) and} (c), we get the situation
illustrated in \eqref{eq:aSlopeGenericAfterAStay},
but with
\[
\SPV(\insStay_a(\pi)) = 
\begin{cases}
  (\SPV(\pi)\setminus\{l\})\cup \{n\}  & \text{ if } e_m<k, \; e_m < s_{m+1}, \; l< s_{m+1} \\
  %This case is not possible
%   (\SPV(\pi)\setminus\{e_m,l\})\cup \{n\} & \text{ if } e_m<k, \; e_m > s_{m+1}, \; l< s_{m+1} \\
  (\SPV(\pi)\setminus\{l\})\cup \{n\} & \text{ if } e_m>k, \; e_m > s_{m+1}, \; l< s_{m+1} \\
  % c-cases below
  (\SPV(\pi)\setminus\{l\})\cup \{e_m,n\} & \text{ if } k<e_m < s_{m+1}, \; l< s_{m+1} \\
  \SPV(\pi)\cup \{e_m,n\} & \text{ if } k< e_m < s_{m+1}, \; l> s_{m+1}. 
\end{cases}
\]
In order for $\pi$ to be $a$-slope-admissible, we must still have $l$
as a peak-value, and $n$ as the only additional peak-value--this is not the
case in any of the outcomes.

\medskip 
\noindent
\textbf{Proof of statement 3}: 
We shall now show that if $\pi$ is in the swap-image,
then $\pi$ is not $a$-slope-admissible.

Suppose $\pi = \insSwap_a(\sigma)$ for some 
non-$a$-peak-admissible $\sigma \in \symS_{n-1}$.
The runs of $\pi$, according to \eqref{eq:insSwapDef}, are
\begin{align*}
&\dotsc
 \underset{\alpha}{ \runBlock{}{}{ak} \shortBlock{}{e'_r}} \;
 \underset{\beta}{ \runBlock{b}{}{} } \;
 \dotsc  \;
 \underset{\gamma_1}{ \shortBlock{s'_r}{} }  \\
&\text{or} \\
&
\dotsc
 \underset{\alpha}{ \runBlock{}{}{ak} \shortBlock{}{e'_m}} \;
 \underset{\beta}{ \runBlock{b}{}{} } \;
 \dotsc  \;
 \underset{\gamma_1}{ \shortBlock{s'_m}{} }
 \runBlock{s'_{m+1}}{}{e'_{m+1}}
 \dotsc  \;
\end{align*}
where $e'_m < s'_{m+1}$ and the last entry of $\gamma_1$ is less than $k$.
We now must verify conditions (i')--(iv') for $\pi$.
The first three conditions are easy to verify. 
Now, if $\gamma_1$ is the last run of $\pi$, we have that 
case (a) of (iv') is satisfied (there is no run of $\pi$
where the first value exceeds $k$).

If $\gamma_1$ is not the last run, we know that $e'_m<s'_{m+1}$,
so $l=e'_m < s'_{m+1}=s_{m+1}$ and thus the condition in case (b)
is satisfied.

\medskip 
\noindent
\textbf{Proof of statement 4}: 
We shall now refine the previous and see exactly what 
is possible for $\pi$ being non-$a$-slope-admissible 
and also not in the swap-image.

A more careful case-by-case analysis as before,
shows that $\pi$ is in the swap-image only if 
it satisfies the first of the four cases listed in (b)--(c).
We can then conclude Statement 4.
\end{proof}

Let $\eeSet^j_{n,a}$ be the set of permutations $\pi \in \symS_{n-1}$ such 
that
\begin{itemize}
 \item $a \in \slopeSet(\pi)$,
 \item $\pi$ is not $a$-slope-admissible, and 
 \item $\pi$ satisfies Case~($j$) of Statement~4 in \cref{lem:aAdmissChar}.
\end{itemize}
Note that $\eeSet^1_{n,a},\dotsc,\eeSet^5_{n,a}$ are disjoint sets.

\begin{example}\label{ex:eeSets}
For example, for $n=7$ and $a=3$,
we have that $\eeSet^1_{n,a}$ comprises $40$ permutations,
including $1235647$ and $4735621$.
Moreover, we have that
\begin{align*}
\eeSet^2_{n,a} &= \{7134526, 7261345\}, \\
\eeSet^3_{n,a} &= \{7134625, 7251346 \}, \\
\eeSet^4_{n,a} &= \{1345276, 2761345, 6134527, 6271345 \}, \\
\eeSet^4_{n,a} &= \{1347625, 2513476, 6134725, 6251347 \}. \\
\end{align*}
\end{example}
We shall see below that $|\eeSet^2_{n,a}|=|\eeSet^3_{n,a}|$ and
$|\eeSet^4_{n,a}|=|\eeSet^5_{n,a}|$ in general.

\begin{definition}\label{def:flip}
Suppose $\pi \in \symS_{n-1}$ belongs to some $\eeSet^j_{n,a}$,
and as always, let $k$ be the element succeeding $a$.
We then know that there is at least one run $s_m,\dotsc,e_m$ of $\pi$,
such that $s_m < k < e_m$, so choose $m$ which maximizes $s_m$.
Hence, the runs of $\pi$ are
\begin{equation}
\dotsc
\underset{\alpha}{ \shortBlock{}{a} \singleBlock{k}\shortBlock{}{l} }
 \dotsc  \;
\underset{\gamma}{ \runBlock{s_{m}}{}{e_{m}} }
 \dotsc  \;
\end{equation}
where $l>k$ is the largest entry in the run containing $a$.
As in the construction of $\insSwap_a$,
we write the $\gamma$-run as a concatenation of the elements less than $k$,
and elements greater than $k$:
\[
 \runBlock{s_{m}}{}{e_{m}} = \shortBlock{s_{m}}{}\shortBlock{}{e_{m}}.
\]
Now, let $\mu$ denote the contiguous sequence of runs (possibly empty)
appearing immediately after $l$ in $\pi$, starting with 
an element greater than $e_m$. 
Similarly, let $\nu$ be the contiguous sequence of runs (possibly empty)
appearing immediately after $e_m$ in $\pi$,
starting with an element greater than $l$.

Hence, the elements in $\alpha$ which are strictly greater than $k$,
together with the elements in $\mu$ form a contiguous segment of 
elements in $\pi$. Let us call this segment $\Sigma_1$.
Moreover, the elements in $\gamma$ greater than $k$,
together with the elements in $\nu$ also form a contiguous segment, $\Sigma_2$.
We then define \defin{$\insFlip_a(\pi)$} as the permutation obtained from $\pi$
by interchanging $\Sigma_1$ and $\Sigma_2$.
It is straightforward to verify that $\insFlip_a$ is an involution.

Finally, we note that the runs of $\insStay_a(\insFlip_a(\pi))$
are 
\begin{equation}\label{eq:runsOfStayAFlip}
\dotsc
 ,\; \shortBlock{}{an},  \;
 \dotsc,  \;\;
\shortBlock{s_{m}}{} \shortBlock{}{l}, \;\; \singleBlock{k} \shortBlock{}{ e_{m} }, \;
 \dotsc  \;.
\end{equation}
\end{definition}

\begin{example}
Consider a permutation $\pi = 83724561$, with $a = 2$.
Then $k = 4$, and we have runs $\alpha = 2456$ and $\gamma = 37$.
Moreover, $\mu_1=\nu_2=\emptyset$.
So, $\insFlip_2(\pi)$ is given by interchanging $\Sigma_1 = 56$ and $\Sigma_2 = 7$ in $\pi$,
which gives $\insFlip_2(\pi) = 83562471$.
\medskip 

Consider now instead $\pi = 134265$, with $a=1$.
Here, $k=3$, $\mu=\emptyset$ and $\nu=5$, $\Sigma_1 = 4$, $\Sigma_2 = 65$.
We have that $\insFlip_1(\pi) = 136524$.
\end{example}
Additional examples are shown in \cref{tab:flip}.

\begin{lemma}\label{lem:flip}
The map $\insFlip_a$, acts as follows:
\begin{align*}
\insFlip_a: &\eeSet^1_{n,a} \to \eeSet^1_{n,a}, \\
\insFlip_a: &\eeSet^2_{n,a} \to \eeSet^3_{n,a}, 
	&\insFlip^{-1}_a: \eeSet^3_{n,a} \to \eeSet^2_{n,a}, \\
\insFlip_a: &\eeSet^4_{n,a} \to \eeSet^5_{n,a}, 
	&\insFlip^{-1}_a: \eeSet^5_{n,a} \to \eeSet^4_{n,a}.
\end{align*}
Moreover, for each map $\psi \in \{\insFlip_a, \insFlip^{-1}_a \} $ above,
\[
\SPV(\pi) \cup \{n\} = \SPV(\insStay_a(\psi(\pi))),
\]
on the relevant domain, and $\psi$ preserves the set of descent bottoms.
\end{lemma}
\begin{proof}
It is fairly straightforward to verify that the domain and range match,
by examining the five cases in Statement~4 of \cref{lem:aAdmissChar}.
Moreover, we have that $\mu=\nu=\emptyset$ (as in the definition of $\insFlip_a$) unless 
we act on some permutation in $\eeSet^4_{n,a}$ or $\eeSet^5_{n,a}$.
\medskip

The case $\insFlip_a: \eeSet^1_{n,a} \to \eeSet^1_{n,a}$ follows the same structure 
as the other two cases, so we only consider those.

We first verify $\insFlip_a: \eeSet^2_{n,a} \to \eeSet^3_{n,a}$.
Let $\pi \in \eeSet^2_{n,a}$, where the runs are
\begin{equation}
\dotsc
\underset{\alpha_1}{\shortBlock{}{}}
 \singleBlock{a}
 \underset{\alpha_2}{\shortBlock{k}{l}} \;
 \dotsc  \;
 \runBlock{s_{m}}{}{e_{m}} 
 \runBlock{s_{m+1}}{}{e_{m+1}} 
 \dotsc 
\end{equation}
and we know that by the choice of $m$ and
properties of $\eeSet^2_{n,a}$, that $s_m<k<l<e_m<s_{m+1}$.
Hence, $l \in \SPV(\pi)$ and $e_m \notin \SPV(\pi)$.
Now as we saw in \eqref{eq:runsOfStayAFlip},
the runs of $\pi' \coloneqq \insStay_a(\insFlip_a(\pi))$ are
\begin{equation}
\dotsc
\shortBlock{}{an}  
\dotsc  \;\;
\shortBlock{s_{m}}{} \shortBlock{}{l} \;\; \singleBlock{k} \shortBlock{}{ e_{m} }
\runBlock{s_{m+1}}{}{e_{m+1}} 
\dotsc  \;.
\end{equation}
We therefore have that $n,l \in \SPV(\pi')$ but $e_m \notin \SPV(\pi')$.
All other peak-values are preserved,
and it is verified that $\SPV(\pi) \cup \{n\} = \SPV(\insStay_a(\insFlip_a(\pi)))$.

\medskip 
Now for the case $\insFlip_a: \eeSet^4_{n,a} \to \eeSet^5_{n,a}$.
For $\pi \in \eeSet^2_{n,a}$, the runs will satisfy
$k<l<s_m<e_m<s_{m+1}$ and the exact same reasoning \emph{for the runs of $\pi$}
is valid as in the previous case.
However, now it is possible that $\insFlip_a$ does 
something involing non-empty $\mu$ and $\nu$ in \cref{def:flip},
but by construction, this more complicated situation will 
still result in $\insStay_a(\insFlip_a(\pi))$ being of the form in
\eqref{eq:runsOfStayAFlip}.
One can now verify again that $\SPV(\pi) \cup \{n\} = \SPV(\insStay_a(\insFlip_a(\pi)))$.

Finally, it is straightforward to see that the set of descent bottoms are preserved.
\end{proof}

In \cref{tab:flip}, we illustrate the action of $\insFlip_1$ on some permutations
in $\eeSet^j_{7,1}$ for different choices of $j$.
\begin{table}
\begin{tabular}{lclc}
$\pi$ & Case & $\insFlip_1(\pi)$ & Case \\
\toprule
2351467 & 1 & 2367145 & 1 \\
2571346 & 1 & 2461357 & 1 \\
1473526 & 1 & 1453726 & 1 \\
2467135 & 1 & 2513467 & 1 \\
\midrule

7135246 & 2 & 7134625 & 3 \\
7134526 & 2 & 7136245 & 3 \\
7625134 & 2 & 7624135 & 3 \\
7362145 & 2 & 7352146 & 3 \\
\midrule
 
1342576 & 4 & 1357624 & 5 \\
6134275 & 4 & 6137524 & 5 \\
2476135 & 4 & 2513476 & 5 \\
7265134 & 4 & 7241365 & 5 \\
1342756 & 4 & 1375624 & 5 \\
2657134 & 4 & 2413657 & 5 \\
3762145 & 4 & 3521476 & 5 \\
2761345 & 4 & 2451376 & 5 \\
\bottomrule
\end{tabular}
\caption{Action of $\insFlip_1$ on various members of $\eeSet^j_{7,1}$,
for different choices of $j \in \{1,2,4\}$ (second column).
The last column tells us in which set $\eeSet^j_{7,1}$ where $\insFlip_1(\pi)$ belongs.
As we see, this agrees with the statements in \cref{lem:flip}.
}\label{tab:flip}
\end{table}

The following result is now an analog of \cref{lem:insertNforPeaks},
and is the main result in this section. The reader is encouraged to
review \cref{fig:fiveCasesBijFig} at this point.
\begin{theorem}\label{thm:insertNforPeaksLex}
For any $n\geq 1$, there is a bijection 
\[
  \peakBijLex : \{\emptyset,1,2,\dotsc,n-1 \} \times \symS_{n-1} \to  \symS_{n}
\]
which has the following properties.
For simplicity, we set $\pi' \coloneqq \peakBijLex(a,\pi)$ and 
we let $k$ be the value immediately
succeeding $a$ in $\runsort(\pi)$, unless $a$ is the last entry in $\runsort(\pi)$:
\begin{enumerate}
 \item $a=\emptyset$, and $\SPV(\pi') = \SPV(\pi)$.
 
 \item $a$ is the last entry of $\runsort(\pi)$, and $\SPV(\pi') = \SPV(\pi)$.
 
 \item \label{en:insertNafterPeakLex} $a \in \SPV(\pi)$. Then 
 \[
  \SPV(\pi') = (\SPV(\pi) \setminus \{a\}) \cup \{n\}.
 \]
 \item \label{en:insertNbeforePeakLex}  $k \in \SPV(\pi)$. Then 
 \[
  \SPV(\pi') = (\SPV(\pi) \setminus \{k\}) \cup \{n\}.
 \]
 
 \item $a$ is not the last entry of $\runsort(\pi)$,
 and neither $a$ or $k$ are in $\SPV(\pi)$. Then
 \[
  \SPV(\pi') = \SPV(\pi)  \cup \{n\}.
 \]
\end{enumerate}
\end{theorem}
\begin{proof}
We shall now describe how $\pi' \coloneqq \peakBijLex(a,\pi)$ is computed
in each of the cases listed above.

\medskip 
\noindent
\textbf{Case 1:} $a=\emptyset$ and $\pi' = \insStay_a(\pi)$.
It is straightforward to see that $\SPV(\pi') = \SPV(\pi)$.

\medskip 
\noindent
\textbf{Case 2:} $a$ is the last entry of $\runsort(\pi)$ and $\pi' = \insStay_a(\pi)$.
Again, it is easy to verify that $\SPV(\pi') = \SPV(\pi)$.

\bigskip 

In all remaining cases, $a$ appears in some run $\alpha$.
Let $k$ be the entry immediately after $a$
in $\runsort(\pi)$. We note that at most one of $a$ and $k$ can be a member of $\SPV(\pi)$.
In the case when $\alpha$ is not the last run,
let $\beta$ be the run immediately succeeding $\alpha$.

\medskip 
\noindent
\textbf{Case 3:} $a \in \SPV(\pi)$, and we know that $a>k$.
We have that
\begin{equation}\label{eq:lexsortPiRuns3}
 \dotsc
 \underset{\alpha}{ \runBlock{}{}{a} } \;
 \underset{\beta}{ \runBlock{k}{}{} } \;
 \dotsc  \;
 \runBlock{s_r}{}{e_r},
\end{equation}
and we set $\pi' = \insStay_a(\pi)$.
The runs of $\pi$ can easily be identified with the runs of $\insStay_a(\pi)$,
and hence
\[
\SPV(\pi') = (\SPV(\pi) \setminus \{a\}) \cup \{n\}.
\]

\medskip 
\noindent
\textbf{Case 4:} $k\in \SPV(\pi)$, so $\alpha$ is for sure not the lex-largest run of $\pi$.
We are therefore in the situation
\begin{equation}\label{eq:lexsortPiRuns4}
 \dotsc
 \underset{\alpha}{ \runBlock{}{}{ak} } \;
  \underset{\beta}{ \runBlock{b}{}{} } \;
 \dotsc  \;
 \runBlock{s_r}{}{e_r}
\end{equation}
where $k>b$. If $\pi$ is $a$-peak-admissible, we let $\pi' = \insStay_a(\pi)$,
otherwise we set $\pi' = \insStay_a(\insSwap_a(\pi))$.
By definition of $a$-peak-admissible or by \cref{lem:swap}, we have that
\[
  \SPV(\pi') = (\SPV(\pi) \setminus \{k\}) \cup \{n\}.
\]

\medskip 
\noindent
\textbf{Case 5:}  $a \notin P$ and $k \notin P$.

\begin{enumerate}[label={(\alph*)}]
 \item If $\pi$ is $a$-slope-admissible, we set $\pi' = \insStay_a(\pi)$.
 Note that by \cref{lem:aAdmissChar}, $\pi'$
 will not be equal to some $\insStay_a(\insSwap_a(\sigma))$,
 from Case~4.
 
 \item If $\pi$ is in the swap-image, so that 
 $\pi$ is of the form $\insSwap_a(\sigma)$ for some (unique) $\sigma$, 
we let
\[
 \pi'  = \insStay_a(\sigma) = \insStay_a(\insSwap^{-1}_a(\pi)).
\]
To elaborate more on this choice, 
suppose that the runs of $\insStay_a(\pi)$ are of the form
of some $\insStay_a(\insSwap_a(\sigma))$ described earlier in \eqref{eq:insSwapDef}:
\begin{equation}\label{eq:insSwapUse}
\dotsc
 \underset{\alpha}{ \runBlock{}{}{an} } \;
 \singleBlock{k}
 \underset{\gamma_2}{\shortBlock{}{e_m}} \;
 \underset{\beta}{ \runBlock{b}{}{} } \;
 \dotsc  \;
 \underset{\gamma_1}{ \shortBlock{s_m}{} }
 \runBlock{s_{m+1}}{}{e_{m+1}}
 \dotsc  \;
\end{equation}
This means that $\runsort(\pi)$ is of the form
\begin{equation*}
\dotsc
 \underset{\alpha}{ \runBlock{}{}{ak} \shortBlock{}{e_m}} \;
 \underset{\beta}{ \runBlock{b}{}{} } \;
 \dotsc  \;
 \underset{\gamma_1}{ \shortBlock{s_m}{} }
 \runBlock{s_{m+1}}{}{e_{m+1}}
 \dotsc  \;
\end{equation*}
where $a$, $k$ and $e_m$  are in the same run, and $e_m \in \SPV(\pi)$.
We can compute $\sigma = \insSwap^{-1}_a(\pi)$
and the runs of $\sigma$ are
\begin{equation}
\dotsc
 \underset{\alpha}{ \runBlock{}{}{ak} } \;
 \underset{\beta}{ \runBlock{b}{}{} } \;
 \dotsc  \;
 \underset{\gamma_1}{\shortBlock{s_m}{}}
 \underset{\gamma_2}{\shortBlock{}{e_m}} \;
 \runBlock{s_{m+1}}{}{e_{m+1}}
 \dotsc 
\end{equation}
Finally, the runs of $\insStay_a(\sigma)$ are lexicographically
sorted as
\begin{equation}
\dotsc
 \underset{\alpha}{ \runBlock{}{}{an} } \;
 \underset{\beta}{ \runBlock{b}{}{} } \;
 \dotsc  \;
 \runBlock{s_{m}}{}{e_{m}} \;
 \singleBlock{k}
 \runBlock{s_{m+1}}{}{e_{m+1}}
 \dotsc  \;,
\end{equation}
and it is now evident that $\SPV(\pi') = \SPV(\pi) \cup \{n\}$,
as desired.

\item The only remaining case is 
when $\pi \in \eeSet^1_{n,a} \cup \dotsb \cup \eeSet^5_{n,a}$.
We then let $\pi' = \insFlip_a(\pi)$, and by \cref{lem:flip},
we again have that $\SPV(\pi') = \SPV(\pi) \cup \{n\}$.
\end{enumerate}

In each case, we can easily determine $a$ from $\pi'$, as this is the entry to the left of $n$.
All the non-trivial maps used, i.e.,~$\insFlip_a$ and $\insSwap_a$, are invertible,
so it follows that the description above indeed defines a bijection.
\end{proof}

The bijections $\peakBij$ and $\peakBijLex$ behave in
a similar manner on other combinatorial statistics.
Now we do need the distinction between Case~5a and Case~5b 
(separating these cases is compatible with $\insFlip_a$,
since by \cref{lem:flip}, $\insFlip_a$ preserves the set of descent bottoms).

In what follows, by convention, if $a=\emptyset$, set $k \coloneqq \pi_1$.
We refer to the introduction where all permutation statistics are defined.
\begin{lemma}\label{lem:otherStatistics}
Let $\sigma \in \symS_{n-1}$ and let $\sigma' \coloneqq \peakBij(\sigma,a)$
for $a \in \{\emptyset,1,2,\dotsc,n-1\}$ (and we let $k$ be the element succeeding $a$ in $\sigma$).
Then for each of the cases in \cref{lem:insertNforPeaks}, we have the following
recursive relations:
\begin{center}
\begin{tabular}{lllll}
Case & $\sigma'_1$ & $\DB(\sigma')$ & $\LTRMIN(\sigma')$ & $\PKacb(\sigma')$ \\
\toprule 
1 & $n$ & $\DB(\sigma) \cup \{k\}$ & $\LTRMIN(\sigma) \cup \{n\}$ & $\PKacb(\sigma)$  \\
2 & $\sigma_1$ & $\DB(\sigma)$ & $\LTRMIN(\sigma)$ & $\PKacb(\sigma)$  \\
3 & $\sigma_1$ & $\DB(\sigma)$ & $\LTRMIN(\sigma)$ & $\PKacb(\sigma)$  \\
4 & $\sigma_1$ & $\DB(\sigma) \cup \{k\}$ & $\LTRMIN(\sigma)$ & $\PKacb(\sigma)$  \\
5a & $\sigma_1$ & $\DB(\sigma) \cup \{k\}$ & $\LTRMIN(\sigma)$ & 
$(\PKacb(\sigma) \setminus \{a\}) \cup \{n\}$ \\
5b & $\sigma_1$ & $\DB(\sigma)$ & $\LTRMIN(\sigma)$ &
$\PKacb(\sigma) \setminus \{a\}$ \\
\bottomrule
\end{tabular}
\end{center}
The same identities hold in the corresponding
cases if we instead take $\sigma' \coloneqq \peakBijLex(\sigma,a)$.
\end{lemma}
\begin{proof}[Proof sketch]

We need to examine both $\peakBij(\cdot,a)$ and $\peakBijLex(\cdot,a)$,
and make sure that they affect the statistic in the same
manner for each of the listed cases above.

\emph{Observation:}
Note that $k$ is always the number succeeding $n$ in $\sigma'$,
for both $\peakBij(\cdot,a)$ and $\peakBijLex(\cdot,a)$.
To show this, one needs to go through all cases in \cref{thm:insertNforPeaksLex},
but it suffices to verify that $\insSwap_a(\sigma)$ and $\insFlip_a(\sigma)$
always have $k$ succeeding $a$.

\medskip 
\noindent
\textbf{First entry statistic.} It is clear that 
if we insert $n$ after any entry $a$, the very first entry of the permutation is preserved.
This is the case also for the more complicated map $\peakBijLex(\cdot,a)$.
Likewise, inserting $n$ in the very beginning makes it the first entry.

\medskip 
\noindent
\textbf{Descent bottom statistic.}
It is straightforward to verify the effect on the 
set of descent-bottoms when we simply apply $\insStay_a(\sigma)$.
So then it remains to verify the same five cases for $\peakBijLex(\sigma,a)$,
but this follows more or less directly from the observation above.

Now, together with \cref{lem:casesCount},
and the fact that the behavior of descent bottoms under $\peakBijLex(\cdot,a)$ 
agrees with $\peakBij(\cdot,a)$, implies that the number of values of $a$ belonging to Case~5a,
is the same for both maps.

\medskip 
\noindent
\textbf{Left to right minima statistic.}
For $\peakBij(\cdot,a)$, we note that inserting $n$ somewhere 
in $\sigma$ does not affect the left to right minima, except when $n$
is inserted in the very beginning. 

Now for $\peakBijLex(\cdot,a)$, note that 
the maps $\insSwap_a(\sigma)$ only swap postfix segments of runs.
Examining $\insFlip_a(\sigma)$ is slightly more intricate,
but it is not difficult to verify that the $\Sigma_1$ and $\Sigma_2$
which are swapped never contain any left to right minima values. 
Since the Swap map $\insSwap_a$ is an involution, this gives the desired property.

\medskip 
\noindent
\textbf{$132$-peak-values.} This is also treated via case-by-case analysis.

\end{proof}

\subsection{A bijection on permutations}\label{ssec:conclusion}

The recursions in \cref{lem:insertNforPeaks} and \cref{thm:insertNforPeaksLex}
have the same structure, and thus allows us to construct an implicit bijection,
\[
  \eta : \symS_n \to \symS_n,
\]
with the main property that $\PEAKVAL(\sigma) = \PEAKVAL(\runsort(\eta(\sigma))$.
The construction is recursive, and it is not canonical.
For example, we shall impose the additional constraint that $\eta$
also preserves the set of descent bottoms, which is possible due to 
\cref{lem:otherStatistics}.

We illustrate this idea with a concrete example.

Suppose $\eta(\sigma) = \sigma'$ for $\sigma\in\symS_{n-1}$,
where we know $\PEAKVAL(\sigma) = \PEAKVAL(\runsort(\eta(\sigma))$,
and $\DB(\sigma) = \DB(\eta(\sigma))$,
For example, with $\sigma=641325$ and $\sigma'=645132$,
suppose we have already determined that
\[
\eta({6, 4, 1, 3, 2, 5}) = ({6, 4, 5, 1, 3, 2}).
\]
By induction, we know that $\sigma$ and $\sigma'$ have the same number of values of $a$
belonging to each respective case of \cref{lem:insertNforPeaks} and \cref{thm:insertNforPeaksLex},
and they have the same set of descent bottoms, ($\{1,2,4\}$ in this case).
In our example, we have the following situation:
\begin{center}
\begin{tabular}{ccc}
Case  & $a$ for $\sigma$ & $a$ for $\sigma'$ \\
\toprule
1			& $\emptyset$ & $\emptyset$ \\ 
2			& $5$ & $6$ \\
3			& $3$ & $3$ \\
4			& $1$ & $1$ \\
5a		& $2$ & $4$ \\
5b		& $4,6$ & $2,5$
\end{tabular}
\end{center}
For example, $a=3$ belongs to the second case for both $\sigma$ and $\sigma'$.
From the first five rows, we do not have any choice, and we set 
\begin{align*}
  \eta( \peakBij(\sigma,\emptyset) ) &\coloneqq \peakBijLex(\sigma',\emptyset) \\
  \eta( \peakBij(\sigma,5) ) &\coloneqq \peakBijLex(\sigma',6) \\
  \eta( \peakBij(\sigma,3) ) &\coloneqq \peakBijLex(\sigma',3) \\
  \eta( \peakBij(\sigma,1) ) &\coloneqq \peakBijLex(\sigma',1) \\
  \eta( \peakBij(\sigma,2) ) &\coloneqq \peakBijLex(\sigma',2)
\end{align*}
However, in the last row, any combination would ensure that 
$\eta$ behaves as desired on peaks and preserves the set of descent bottoms. 
For example, we can take
\begin{align*}
  \eta( \peakBij(\sigma,2) ) &\coloneqq \peakBijLex(\sigma',4) \\
  \eta( \peakBij(\sigma,4) ) &\coloneqq \peakBijLex(\sigma',2) \\
  \eta( \peakBij(\sigma,6) ) &\coloneqq \peakBijLex(\sigma',6).
\end{align*}
In conclusion, we have 
\begin{align*}
 \eta(7, 6, 4, 1, 3, 2, 5) &= ( {7, 6, 4, 5, 1, 3, 2} ) \\
 \eta(6, 4, 1, 3, 2, 5, 7) &= ( {6, 7, 4, 5, 1, 3, 2} ) \\
 \eta(6, 4, 1, 3, 7, 2, 5) &= ( {6, 4, 5, 1, 3, 7, 2} ) \\
 \eta(6, 4, 1, 7, 3, 2, 5) &= ( {6, 4, 5, 1, 7, 3, 2} ) \\
 \eta(6, 4, 1, 3, 2, 7, 5) &= ( {6, 4, 7, 5, 1, 3, 2} ) \\
 \eta(6, 4, 7, 1, 3, 2, 5) &= ( {6, 4, 5, 1, 3, 2, 7} ) \\
 \eta(6, 7, 4, 1, 3, 2, 5) &= ( {6, 7, 4, 5, 1, 3, 2} ),
\end{align*}
and thus, we have recursively extended the domain of $\eta$.
We encourage the reader to verify that peak-values in the left hand side,
are mapped to peak-values after runsort in the right hand side.

In general, there is no choice in Case~1 and Case~2, as there is always exactly one 
possible value of $a$ there. In Case~3 and Case~4, we have many possible combinations
if we are only interested in preserving the \emph{number} of peaks,
but the choice is unique if we want to map $\PEAKVAL$ to $\SPV$. In Case~3 for example, 
if $n$ is inserted after a peak-value $k$, then this value $k$ must be the same on both sides
(but correspond to different values of $a$).
But for Case~5a and Case~5b, there is some freedom, even if we demand 
that the set of descent bottoms is preserved.
Incorporating additional statistics from \cref{lem:otherStatistics}
will not reduce this freedom further. Rather, it is necessary to
find some combinatorial statistic which split Case~5a and Case~5b further,
or make the pairing canonical.

\begin{example}
Here are the values of $\eta$, constructed with the above method.
It turns out that this construction is canonical for the values presented below,
but in general, choices have to be made.
\begin{center}
\begin{tabular}{cc}
$\sigma$&$\eta(\sigma)$ \\
\toprule
12 & 12 \\
21 & 21 \\
\midrule 
123 & 123 \\
132 & 132 \\
213 & 231 \\
231 & 213 \\
312 & 312 \\
321 & 321 \\
\end{tabular}
\quad
\begin{tabular}{cc}
$\sigma$&$\eta(\sigma)$ \\
\toprule
 1234 & 1234 \\
 1243 & 1243 \\
 1324 & 1324 \\
 1342 & 1342 \\
 1423 & 1423 \\
 1432 & 1432 \\
 2134 & 2341 \\
 2143 & 2431 \\
 \end{tabular}
\quad
\begin{tabular}{cc}
$\sigma$&$\eta(\sigma)$ \\
\toprule
 2314 & 2413 \\
 2341 & 2134 \\
 2413 & 2314 \\
 2431 & 2143 \\
 3124 & 3412 \\
 3142 & 3142 \\
 3214 & 3421 \\
 3241 & 3214 \\
\end{tabular}
\quad
\begin{tabular}{cc}
$\sigma$&$\eta(\sigma)$ \\
\toprule
 3412 & 3124 \\
 3421 & 3241 \\
 4123 & 4123 \\
 4132 & 4132 \\
 4213 & 4231 \\
 4231 & 4213 \\
 4312 & 4312 \\
 4321 & 4321 \\
\end{tabular}
\end{center}
\end{example}

The discussion above (relying on \cref{lem:otherStatistics})
can be formulated as the following theorem which is the main result in this section.
\begin{theorem}\label{thm:mainPeakIdentity}
For $n \geq 1$, we have that 
\begin{equation*}
 \sum_{\pi \in \symS_n}
 \xvec_{\DB(\pi)}
 \yvec_{\LTRMIN(\pi)}
 \zvec_{\PKacb(\pi)}
 \wvec_{\PEAKVAL(\pi)} 
=
 \sum_{\pi \in \symS_n} 
 \xvec_{\DB(\pi)}
 \yvec_{\LTRMIN(\pi)}
 \zvec_{\PKacb(\pi)}
 \wvec_{\SPV(\pi)}.
\end{equation*}
\end{theorem}
\begin{proof}
Apply any of the bijections $\eta$, as constructed above,
which preserves $\DB$, $\LTRMIN$ and $\PKacb$, and sends $\PEAKVAL$ to $\SPV$.
\end{proof}

We end this section with the following proposition,
giving a new interpretation of \oeis{A202365}.
\begin{proposition}
The numbers $B_{n}$ of permutations $\sigma$ such that $\runsort(\sigma)$ 
has a descent at position $2$ satisfies the recurrence relation 
\begin{equation*}
B_{n} = (n-1)! + (n-2)B_{n-1},
\end{equation*} with initial condition $B_{3} = 2$.
Consequently, $B_{n} = (n-2)!(n+1)\dfrac{n-2}{2}$ for $n \geq 3$.
These also correspond to the counting sequence \oeis{A202365}
whose first few values are 
\[
2, 10, 54, 336, 2400, 19440, 176400, \dotsc.
\]
\end{proposition}
\begin{proof}
Consider $\sigma \in \symS_{n-1}$.
Inserting $n$ immediately after $1$ in $\sigma$
and then run-sorting, gives a permutation $\runsort(\sigma) \in \symS_n$,
whose first run is of the form $1n$.
Hence, we obtain $(n-1)!$ permutations having a descent at position $2$.

We must now count the number of permutations $\sigma$,
where we have a descent at position $2$, and $\sigma(2) \in \{3,4,\dotsc,n-1\}$. 
So now suppose $\sigma \in B_{n-1}$ which already has a descent at position $2$,
and $\sigma(2)<n$. We can then insert $n$ anywhere, except after $1$ or 
after $\sigma(2)$. The resulting permutation $\sigma' \in \symS_n$
will then also have the property that $\runsort(\sigma')$
lies in $B_{n}$.
This method gives $(n-2)B_{n-1}$ additional permutations,
and its straightforward to see that all elements in $B_n$ can be produced
by exactly one of the two methods above.
\end{proof}

\subsection{Probabilistic statements}

It is a standard exercise to compute the expected number of descents in a permutation.
\begin{proposition}
The expected number of descents in a permutation in $\symS_n$ is $(n-1)/2$.
\end{proposition}
\begin{proof}
Let $I_k$ be the indicator variable that position $k \in [n-1]$ is a descent.
We have that $\ee[I_k]=1/2$, since the permutation with entries at $k$ and $k+1$ swapped are equally likely.
Hence, with $\sigma \in \symS_n$,
\[
\ee[\des(\sigma)] = \sum_{k=1}^{n-1} \ee[I_k] = \frac{n-1}{2}.
\]
\end{proof}

Given a uniformly random permutation in $\symS_n$,
it is natural to ask what the expected number of descents is 
after performing $\runsort$. 
We can now answer this in the following theorem.
\begin{theorem}[{See also \cite[p.110]{WarrenSeneta1996}}]
Let  $\sigma \in \symS_n$, with $n\geq 2$ be uniformly chosen.
Then 
\[
\ee[\des(\runsort(\sigma))] = \ee[\peaks(\runsort(\sigma))] 
=  \ee[\peaks(\sigma)] = \frac{n-2}{3}.
\]
\end{theorem}
\begin{proof}
The first equality follows from the fact 
that the number of descents after run-sort is the same as the number of peaks,
as we noted in \cref{rem:desArePeaks}
The second identity follows immediately from \cref{thm:mainPeakIdentity},
so it suffices to compute the expected number of peaks in 
a permutation. 
This has been done before in greater generality, see e.g.,~\cite{FulmanKimLee2019x}.
We include an argument below for the sake of completeness.

Recall from \eqref{eq:peakGenFunc}, the exponential generating function
\[
G(u,t) \coloneqq \sum_{n\geq 0} \frac{u^n}{n!}\sum_{\pi \in \symS_n} t^{\peaks(\pi)}
= \dfrac{\tan(u\sqrt{t-1})}{\sqrt{t-1} - \tan(u\sqrt{t-1})}.
\]
Note that for $\sigma \in \symS_n$, the expectation $\ee[\peaks(\sigma)]$
is equal to the coefficient of $u^n$ in
\[
t \frac{\partial G(u,t)}{\partial t} = \sum_{n\geq 0} \frac{u^n}{n!}\sum_{\pi \in \symS_n} \peaks(\pi) \cdot t^{\peaks(\pi)},
\]
evaluated at $t=1$.
Differentiating $G(t, u)$ with respect to $t$ gives
\begin{equation}\label{generatingfn}
G_{t}(t, u) = \dfrac{t\sqrt{t-1} \sec^2(u\sqrt{t-1})-\tan(u\sqrt{t-1})}{2\sqrt{t-1}(\sqrt{t-1}-\tan(\sqrt{t-1}))^2}.
\end{equation}
Changing variables, with $\theta = u\sqrt{t-1}$, so that 
\[
\sqrt{t-1} = \dfrac{\theta}{u} \qquad \text{ and } 2\sqrt{t-1} = \dfrac{2\theta}{u}.
\]
Then \eqref{generatingfn} becomes 
\begin{equation*}
\dfrac{u^3 (\theta + \theta\tan^2(\theta)-\tan\theta)}{2\theta(\theta^2 - 2\theta u \tan\theta + u^2 \tan^2\theta)}.
\end{equation*}
We now require to find the limit of \eqref{generatingfn} as $t \rightarrow 1$, 
we note that in the change of variables, $\theta = u\sqrt{t-1}$, when $t = 1$, then $\theta = 0$.

Now we obtain 
\begin{equation}\label{generatingfn2}
\lim_{\theta \rightarrow 0}\dfrac{u^3 (\theta + \theta\tan^2(\theta)-\tan\theta)}{2\theta(\theta^2 - 2\theta u \tan\theta + u^2 \tan^2\theta)}, 
\end{equation} which is an indeterminate form $\frac{0}{0}$. 
Applying L'Hopitals rule three times, we obtain that \eqref{generatingfn2} becomes 
\begin{equation}\label{expectedvalue}
\dfrac{4u^3}{2(6-12u+6u^2)} = \frac{u^3}{3(u-1)^2} =
\sum_{n=2}^{\infty}\dfrac{n-2}{3}u^n.
\end{equation}
Since $G(t, u) = \sum_{n \geq 0} B_{n}(t) \frac{u^n}{n!}$, 
we may now compare coefficients, and it follows immediately that
\[
\ee[\peaks(\sigma)] = \frac{n-2}{3},
\]
for $\sigma \in \symS_n$ being uniformly chosen.
\end{proof} 

A result of a similar flavor was recently obtained by C.~Defant~\cite{Defant2020x},
where the expected number of descents after \emph{stack-sorting}
is computed.

\medskip

We end this section with a question.
\begin{question}
Let $\sigma \in \symS_n$ be a uniformly chosen permutation,
and let $\sigma' \coloneqq \runsort(\sigma)$.
We can then consider the permutation matrix with entries equal to $1$ at $(i,\sigma'_i)$, $i \in [n]$.
Typically, this matrix (after scaling) looks like \cref{fig:randomRunSorted},
where a distinctive curve is seen.
As $n\to \infty$ does this curve approach some limit curve?
\begin{figure}[!ht]
\centering
\includegraphics[width=0.4\textwidth]{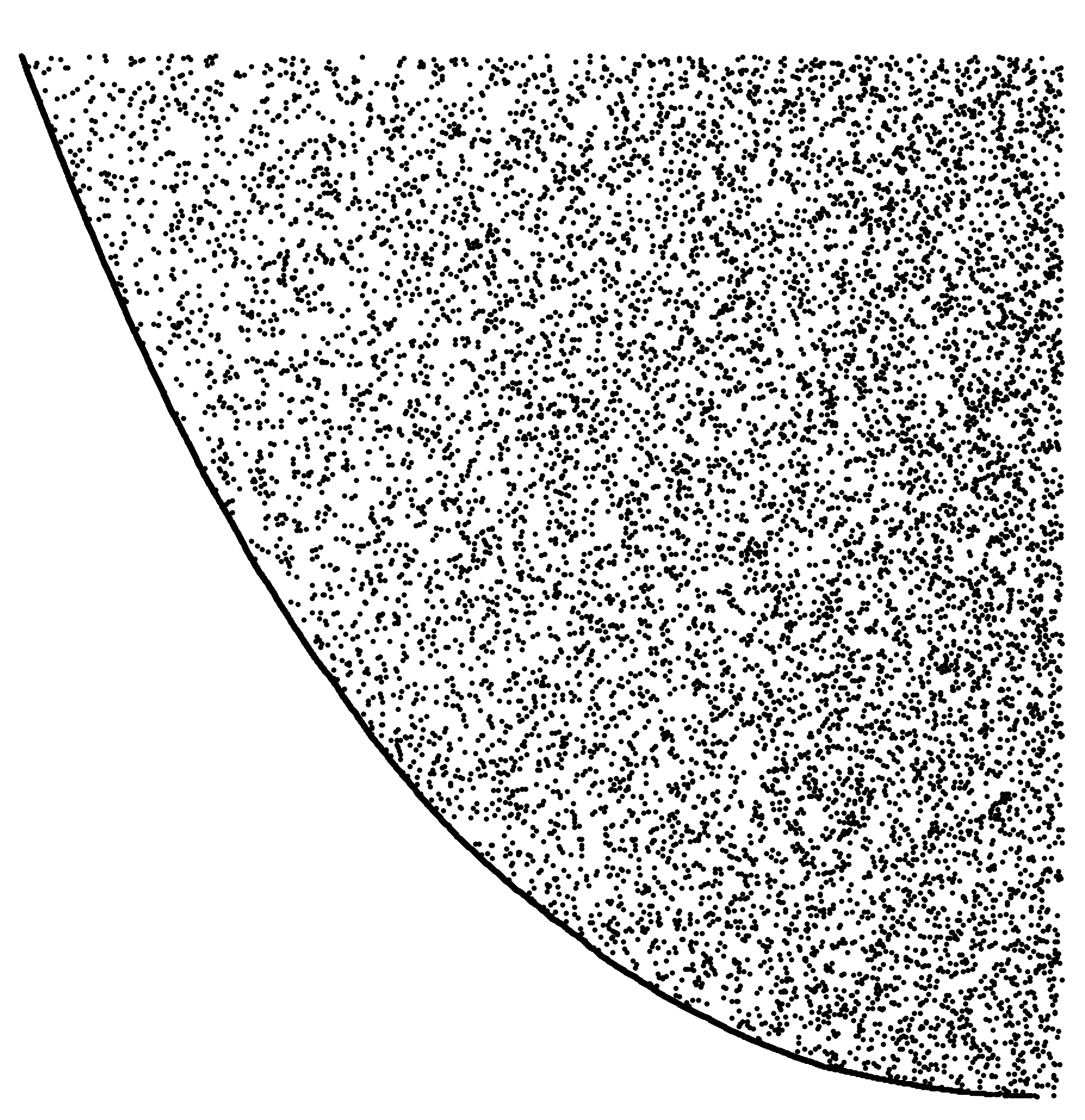} 
\caption{A random permutation matrix $\sigma'$ after lexsort, for $n=20000$.
The entries equal to $1$ are shaded black.}\label{fig:randomRunSorted}
\end{figure}
\end{question}

\section{Run-sorted binary words}\label{D}

Recall that $\RSW(a,b)$ is the set of binary words of length $a+b$,
with $a$ 0s and $b$ 1s, such that the runs
in the word occur in lexicographical order. 
Moreover, recall that the \defin{major index} and the number of \defin{inversions}
of a permutation $\sigma \in \symS_n$ are defined as
\[
\defin{\maj(\sigma)}\coloneqq \sum_{j \in \DES(\sigma)} j, \quad 
\defin{\inv(\sigma)}\coloneqq |\{(i,j): 1 \leq i < j \leq n \text{ and } \sigma(i)>\sigma(j)\}|.
\]
The first result in this section is the following identity.
\begin{proposition}\label{eq:genFuncForBinary}
We have that
\begin{equation}\label{eq:lexbinaryGenFunc}
\sum_{\substack{ a \geq 0 \\ b \geq 0}}  |\RSW(a,b)| q^a t^b = 
 \prod_{i,j \geq 1} \frac{1}{1-q^i t^j}.
\end{equation}
Moreover, for any $a,b \geq 0$, $\RSW(a,b)$ 
is equal to the number of permutations $\sigma$ in $\symS_n$,
such that $\maj(\sigma)=a$, $\maj(\sigma^{-1})=b$.
\end{proposition}
Now, let the \defin{reverse-flip} map $\revflip$,
acting on binary words $w$ of fixed length $n$, be defined as 
\[
  \defin{\revflip(w_1,\dotsc,w_n)} \coloneqq (\hat{w}_n, \hat{w}_{n-1},\dotsc, \hat{w}_2, \hat{w}_1)
\]
where $\defin{\hat{w}_j} \coloneqq 1- w_j$.
Finally, let $\Gamma(w) \coloneqq \runsort(\revflip(w))$.
We show below that $\Gamma: \RSW(a,b) \to \RSW(b,a)$ is a bijection.
Moreover, we can also prove the following.
\begin{proposition}\label{prop:symmetricBW}
There is a bijection between the set
\[
  \{ w \in \RSW(n,n) : \Gamma(w) = w  \}
\]
and $\partitionP(n)$, the set of integer partitions of $n$.
Moreover, for any $n \geq 1$, 
\begin{equation*}
\sum_{\substack{w \in \RSW(n, n) \\ \revflip(w)=w }}
 x^{\beta(w)} = 
\sum_{\lambda \in \partitionP(n)}
 x_1^{\lambda_1} x_2^{\lambda_2} \dotsm x_{\ell}^{\lambda_\ell},
\end{equation*}
where $x^{\beta(w)} = x_1^{\beta_1} \dotsm  x_\ell^{\beta_\ell}$,
with $\beta_j$ denoting the number of $0$s in the $j^\thsup$ run of $w$.
\end{proposition}

\subsection{From run-sorted binary words to biwords}

A \defin{biword} (see \cite{GarsiaGessel1979}, where they are called bipartitions)
is an array with two rows with positive integer entries,
\[
\begin{pmatrix}
\uvec \\ \vvec
\end{pmatrix} = 
\begin{pmatrix}
u_{1} & u_{2} & \cdots & u_{m}\\
v_{1} & v_{2} & \cdots & v_{m}
\end{pmatrix},
\]
subject to the following conditions:
\begin{itemize}
 \item the entries in the top row are weakly increasing;
 \item if $u_i = u_{i+1}$ then $v_i \leq v_{i+1}$.
\end{itemize}
Let $\BW(a,b)$ be the set of biwords where 
the top row has total sum $a$, and the bottom row has total sum $b$.

We shall now describe a bijection from $\RSW(a,b)$ to $\BW(a,b)$ as follows.
Let $w \in \RSW(a,b)$ where $w = \rho_1 \rho_2 \dotsb \rho_r$
are the runs of $w$.
For each $i \in [r]$ we have a column
we have a column $\left(\begin{smallmatrix} u_i \\ v_i \end{smallmatrix}\right)$
in the biword, where $u_i$ is the number of $0$s in $\rho_{r+1-i}$
and $v_i$ is the number of $1$s in $\rho_{r+1-i}$.

We observe the following properties of this bijection.
\begin{itemize}
 \item We have 
 \[
  u_1 + \dotsb + u_r = a, \qquad v_1 + \dotsb + v_r = b.
 \]
 \item The top row is weakly increasing
  and if $u_i = u_{i+1}$ then $v_i \leq v_{i+1}$,
  since the word is run-sorted.
\end{itemize}
It is now straightforward to see that this map is indeed a bijection
between $\RSW(a,b)$ and $\BW(a,b)$.

\begin{example}
Consider a binary word $w = 00011011011101111 \in \RSW(6, 11)$. 
The word $w$ has four runs, namely $00011, 011, 0111$ and $01111$.
The biword corresponding to $w$ is given by
\[
\begin{pmatrix}
1 & 1 & 1& 3\\
4 & 3 & 2& 2
\end{pmatrix}.
\]
\end{example}

The following proposition is well-known, see e.g., \cite{GarsiaGessel1979}.
It is one of the steps to show the classical Cauchy-identity via the Robinson--Schensted--Knuth correspondence,
see \cite[Ch.7]{StanleyEC2}.
\begin{proposition}\label{prop:BWgenFunc}
 We have 
 \[
   \sum_{a,b \geq 0} |\BW(a,b)| q^a t^b = \prod_{i,j \geq 1} \frac{1}{1-q^i t^j}.
 \]
\end{proposition}
\begin{proof}
By comparing coefficients, it is enough to show that 
$|\BW(a,b)|$ is the coefficient of $q^a t^b$
in the product
 \[
  \prod_{i,j \geq 1} 
  \left(
  1 + (q^i t^j) + (q^i t^j)^2 + (q^i t^j)^3 + \dotsb
  \right)
\]
But now this is evident; the term $(q^i t^j)^k$
corresponds to the number of columns in
the biword which are of the form $\binom{i}{j}$.
\end{proof}
By using the bijection between biwords and run-sorted binary words,
we can now deduce \cref{eq:genFuncForBinary}.

\medskip 
We now introduce a set of permutations, 
which we shall see have same cardinality as $\BW(a,b)$.
Let $\defin{\MIP(a,b)} \subset \symS_{a+b}$
be the set of permutations $\sigma$ of $[a+b]$ with $\maj(\sigma)=a$, and $\maj(\sigma^{-1})=b$.
That is,
\begin{equation}\label{eq:majInvMajSum}
\MIP(a,b) \coloneqq 
\{ \sigma \in \symS_{a+b} : \maj(\sigma) =a \text{ and } \maj(\sigma^{-1}) = b \}.
\end{equation}
It is known, see e.g., \cite[p.33]{Foata1977}
that for $n \geq 0$,
\[
  \sum_{\sigma \in \symS_n} q^{\maj(\sigma)} t^{\inv(\sigma)}
  =
  \sum_{\sigma \in \symS_n} q^{\maj(\sigma)} t^{\maj(\sigma^{-1})}
\]
so there are several combinatorial interpretations of the quantity $\MIP(a,b)$.
See also the references in \oeis{A090806}.

In \cite{Roselle1974}, D.~Roselle proved that
 \begin{equation}
 \sum_{n \geq 0} \frac{z^n}{
 (q)_n
 (t)_n
 } 
 \sum_{\pi \in \symS_n} t^{\maj(\pi)} q^{\maj(\pi^{-1})} 
 =
 \prod_{i,j \geq 0} \frac{1}{1-z q^i t^j},
\end{equation}
where $\defin{(q)_n} \coloneqq (1-q)(1-q^2)\dotsb(1-q^n)$.

A similar result was obtained by 
M.~Cheema and T.~Motzkin~\cite[Thm.4.1]{CheemaMotzkin1971}
and later generalized A.~Garsia and I.~Gessel.
In fact, it is possible to deduce the following proposition from Roselle's formula.
\begin{proposition}[{See  \cite[p. 299]{GarsiaGessel1979}
and \cite[Remark 6.2]{CheemaMotzkin1971} }]
We have
\begin{equation}\label{eq:mipGenFunc}
\sum_{\substack{ a \geq 0 \\ b \geq 0}} 
 |\MIP(a,b)| q^a t^b = 
 \prod_{i,j \geq 1} \frac{1}{1-q^i t^j}.
\end{equation}
\end{proposition}
Combining \eqref{eq:mipGenFunc} with \cref{prop:BWgenFunc},
we see that
\[
 \sum_{\substack{ a \geq 0 \\ b \geq 0}} 
 |\BW(a,b)| q^a t^b = 
 \sum_{\substack{ a \geq 0 \\ b \geq 0}}  |\MIP(a,b)| q^a t^b.
\]

\begin{problem}
 Find a bijection between the sets $\RSW(a,b)$ and $\MIP(a,b)$.
\end{problem}

\subsection{Symmetric binary words}

Recall the bijection between run-sorted binary words
and biwords. The following lemma is straightforward to prove.
\begin{lemma}
Suppose $w$ corresponds to the biword $\binom{\uvec}{\vvec}$.
Then $\Gamma(w)$ has the biword $\binom{\vvec'}{\uvec'}$,
which obtained from $\binom{\uvec}{\vvec}$
by first interchanging the two rows, and then sorting the columns 
so that the conditions of being a biword are met.
Consequently, $\Gamma : \RSW(a, b) \to \RSW(b, a)$ is a bijection.
\end{lemma}

\begin{example}
We have $w = 00011011011101111 \in \RSW(6, 11)$,
with biword
\[
\begin{pmatrix}
1 & 1 & 1& 3\\
4 & 3 & 2& 2
\end{pmatrix}.
\]
Then 
$ \Gamma(00011011011101111) = 00001000100100111$
and the latter word has the biword
\[
\begin{pmatrix}
2 & 2 & 3 & 4 \\
3 & 1 & 1 & 1
\end{pmatrix}.
\]
\end{example}

% 
% 
% \begin{example}
% Consider a binary word $w = 001101011 \in \RSW(4, 5)$. Then $\revflip(w) = w^{*} = 001010011$ and $\runsort(001010011) = w' = 001001101 \in \RSW(5, 4)$.
% \end{example}
% In situations where the number of $0$'s from the left is symmetric to the number of $1$'s from the right in the binary word $w$, then $\revflip(w) = w$. Then we have the following result.	

Suppose now $\Gamma(w)=w$ for some $w \in \RSW(n,n)$.
It follows from the bijection with biwords that the biword corresponding to 
$w$ is of the form
\[
\begin{pmatrix}
\lambda_{m} & \lambda_{m-1} & \cdots & \lambda_{1} \\
\lambda_{1} & \lambda_{2} & \cdots & \lambda_{m}
\end{pmatrix},
\]
for some $\lambda \in \partitionP(n)$, where $\lambda_j$
is the number of $0$s in the $j^\thsup$ run of $w$.
For example,
\[
 001010101011 \leftrightarrow 
\begin{pmatrix}
1 & 1 & 1 & 2 \\
2 & 1 & 1 & 1 
\end{pmatrix}
 \leftrightarrow (2,1,1,1) \in \partitionP(5).
\]
This reasoning then allows us to deduce \cref{prop:symmetricBW}.

\subsection{Descents in binary words after run-sorting}

Let $U_k(n)$ be defined as 
\[
  U_k(n) \coloneqq \{w \in \BW(n) :  \des(\runsort(w))=k \}.
\]
\begin{proposition}\label{prop:cake2}
We have that for any $n\geq 1$, $U_0$ is the set of binary words of length $n$
matching the \emph{regular expression} \cite{wiki:regexp} $\mathtt{1{*}0{*}1{*}0{*}}$.
Such words can be split into two cases, matching either
$\mathtt{1{*}0{*}}$ or $\mathtt{1{*}({0+}{1+})0{*}}$.
These are the words of the form
\begin{equation}\label{eq:1010Form}
\begin{cases}
 \underbrace{1\dotsc 1}_{a\geq 0} \; \underbrace{0\dotsc 0}_{d\geq 0}   &a+d=n, \text{ or} \\
 \underbrace{1\dotsc 1}_{a\geq 0} \; \underbrace{0\dotsc 0}_{b\geq 1} \; \underbrace{1\dotsc 1}_{c\geq 1} \; \underbrace{0\dotsc 0}_{d\geq 0} & 
 a+b+c+d=n.
 \end{cases}
\end{equation}
Moreover, for $k\geq 2$, the set $U_{k-1}(n)$ consists of all
binary words of length $n$ matching the regular expression
$\mathtt{{1*}({0+}{1+})\{k\}{0*}}$, i.e., words of the form
\begin{equation}\label{eq:regex}
 \underbrace{1\dotsc 1}_{a\geq 0} \; 
 \underbrace{0\dotsc 0}_{b_1\geq 1} \; \underbrace{1\dotsc 1}_{c_1\geq 1} \; 
 \underbrace{0\dotsc 0}_{b_2\geq 1} \; \underbrace{1\dotsc 1}_{c_2\geq 1} \; 
 \dotsb 
 \underbrace{0\dotsc 0}_{b_k\geq 1} \; \underbrace{1\dotsc 1}_{c_k\geq 1} \;
 \underbrace{0\dotsc 0}_{d\geq 0},
\end{equation}
where $a+(b_1+c_1+b_2+c_2+\dotsb + b_k+c_k)+d = n$.
\end{proposition}
\begin{proof}
First, note that every word in $\BW(n)$ belongs to exactly one of the sets $U_k(n)$.
The binary words with $0$ or $1$ subwords equal to \texttt{01}, are captured in \eqref{eq:1010Form},
while binary words with $k$ subwords of the form \texttt{01} belong to \eqref{eq:regex}.

Hence, it suffices to show that the words in $U_k(n)$ has exactly $k$ descents after run-sort,
which is straightforward.
\end{proof}

Analogous to the results in \cref{C}, we can study the distribution of 
descents in binary words after run-sorting.
\begin{theorem}\label{conj:binWordsRunSort}
Let $C_n(t) \coloneqq \sum_{w \in \BW(n)} t^{\des(\runsort(w))}$.
Then
\begin{equation}\label{eq:BWDesAfterLex}
  C_n(t) =
 \binom{n+1}{3}+(n+1)+
 \sum_{k=2}^n t^{k-1} \left(\binom{n}{2k} + \binom{n}{2k+1}\right).
\end{equation}
In particular, the number of binary words of length $n$ with no descents after run-sort,
is given by \oeis{A000125}, the Cake numbers, $\binom{n+1}{3}+(n+1)$.
\end{theorem}
\begin{proof}
The case $k=0$ is straightforward; there are exactly $n+1$ words of the form \texttt{1*0*}.
That the words in the second case of \eqref{eq:1010Form}
are $\binom{n+1}{3}$ can be proved in the same manner as the $k \geq 2$ situation,
that is, that the number of words of length $n$ of the form in \eqref{eq:regex},
is the sum $\binom{n}{2k} + \binom{n}{2k+1}$.

Observe that for each binary word of the form in \eqref{eq:regex},
we have segments of identical digits of lengths $a \geq 0$, $d \geq 0$, $b_i \geq 1$
and $c_i\geq 1$, for $1 \leq i \leq k$.
For each such segment of identical digits, we need to choose the position 
of the last digit. When $a = 0$, there  are $k$ segments of $0$s and $k$ segments of $1$s, hence giving us a total of $2k$ positions of the last $0$ and $1$ in each of these segments.
The positions of these digits are chosen among the $n$ positions.
Hence we have a total of $\binom{n}{2k}$ words of the form in \eqref{eq:regex} with $a = 0$.

Moreover, when $a \geq 1$, then the first segment contains at least one digit.
Hence we need to choose the position of the last digit for this segment also,
on top of the $2k$ positions from segments of lengths $b_i$ and $c_i$.
Thus we have $2k+1$ positions to be chosen out of $n$ total positions.
This gives $\binom{n}{2k+1}$ binary words of the
form in \eqref{eq:regex} with $a > 0$, and we are done.
\end{proof}

\subsection{Expected number of descents in a run-sorted binary word}\label{ssec:BWEEDes}
  
From the formula in \cref{eq:BWDesAfterLex}
we can easily produce the following exponential generating function:
\begin{align}
H(u,t) &\coloneqq
\sum_{n\geq 0} \frac{u^n}{n!} \left( \binom{n+1}{3}+(n+1)+
 \sum_{k=2}^n t^{k-1} \left(\binom{n}{2k} + \binom{n}{2k+1}\right) \right) \notag \\
&= \frac{e^u \left(\sinh \left(\sqrt{t} u\right) + 
\sqrt{t} \left((t-1) (u+1)+\cosh \left(\sqrt{t} u\right)\right)\right)}{t^{3/2}}
\label{eq:binWordsRunSortGF}.
\end{align}
Mathematica produces the generating function in \eqref{eq:binWordsRunSortGF} 
automatically by using the command
\begin{lstlisting}
FullSimplify[
 Sum[u^n/n!(Binomial[n+1,3]+n+1 + 
  Sum[t^(k-1)(Binomial[n,2k]+Binomial[n,2k+1]),{k,2,n}]),
 {n,0,Infinity}]
]
\end{lstlisting}
From $H(u,t)$, we can compute the expected number of descents in
a binary word after run-sort. It is given by first taking the $t$-derivative 
of $H(u,t)$, then substituting $t=1$. Finally, we extract the coefficient of $u^n$,
multiply by $n!$ and divide by $2^n$ (the total number of binary words). That is,
\[
 \ee[\des(\runsort(w))] = \frac{n!}{2^n} [u^n]H'_t(u,1).
\]
Mathematica can again do all these steps for us.
If we let \texttt{genFunc} denote the exponential generating function
$H(u,t)$, then the command
\begin{lstlisting}
FullSimplify[
 (n!/2^n)*SeriesCoefficient[
  FullSimplify[D[genFunc,t]]/.t->1,
  {u,0,n}]
]
\end{lstlisting}
gives us $\frac{n-5}{4}+2^{-n}(n+1)$.

\section*{Acknowledgement}
O. Nabawanda acknowledges the financial support extended by the Swedish Sida Phase-IV
bilateral program with Makerere University. Special thanks go to J\"orgen~Backelin
and Paul~Vaderlind for all their valuable inputs and suggestions.
We are also thankful for suggestions from Petter Br{\"a}nd{\'e}n, in particular 
the proof of \cref{thm:eulierianMVinterlacing}. We thank M.~Rubey for suggesting to look at left-to-right
minima in \cref{lem:otherStatistics}.

Many thanks to colleagues from CoRS -- Combinatorial Research Studio for lively discussions.

\newpage
\bibliographystyle{alphaurl}
\bibliography{bibliography}

\end{document}